\documentclass[12pt]{article}
\usepackage{amsmath,amsfonts,amssymb,amsthm,anysize,comment}

\def\depth{\mbox{\rm depth} \,}
\def\Proj{\mbox{\rm Proj} \,}

\def\Ker{\mbox{\rm Ker} \,}

\def\qed{$\Box$}

\theoremstyle{plain}
\newtheorem{theorem}{Theorem}[section]
\newtheorem{cor}[theorem]{Corollary}
\newtheorem{prop}[theorem]{Proposition}
\newtheorem{lem}[theorem]{Lemma}

\theoremstyle{definition}
\newtheorem{df}[theorem]{Definition}
\newtheorem{dfprop}[theorem]{Definition and Proposition}
\newtheorem{ex}[theorem]{Example}
\newtheorem{rem}[theorem]{Remark}

\newtheorem{ques}[theorem]{Question}

\begin{document}

\title{Syzygy Theoretic Approach to Horrocks-type Criteria for Vector Bundles}
\author{Chikashi Miyazaki
\thanks{
Partially supported by Grant-in-Aid
for Scientific Research (C) (21K03167)
Japan Society for the Promotion of Science.  
}}
%\date{}
%\date{June 24, 2026}

\maketitle

\begin{abstract}
This paper studies a variant of Horrocks-type criteria for vector bundles
mainly through a syzygy theoretic approach.
Starting with 
sketching splitting criteria for ACM and
Buchsbaum bundles on projective space,  
we try to give new sights of an investigation of
quasi-Buchsbaum bundles.
Not only our result characterizes the null-correlation
bundle on ${\mathbb P}^n$, but also 
classifies
vector bundles on ${\mathbb P}^3$ with
simple intermediate cohomologies
in terms of standard system of parameters and the corresponding
skew-symmetric matrices.
Our approach leads to an extended study
of quasi-Buchsbaum bundles on ${\mathbb P}^n$ with 
small intermediate cohomologies. 
\end{abstract}

\section{Introduction}

The purpose of this paper is to study Horrocks-type criteria for
vector bundles on projective space, especially to clarify
the structure of a vector bundle on ${\mathbb P}^n$
as a quasi-Buchsbaum graded module over the polynomial ring.
Horrocks' celebrated theorem \cite{Ho} says that a vector
bundle on projective space without intermediate cohomologies
is isomorphic to a direct sum of line bundles, which
is based on the categorical equivalence
that a stable equivalence class of vector bundles ${\mathcal E}$
corresponds
to $\tau_{>0}\tau_{<n}{\mathbb R}\Gamma_*({\mathcal E})$.
On the other hand, a Buchsbaum vector bundle on projective space
is obtained to be isomorphic to a direct sum of sheaves of differential
 $p$-form with balanced twist by Chang \cite{C} and Goto \cite{G}.
In order to study the structure of indecomposable quasi-Buchsbaum bundles,
we will construct a free resolution
coming from its nonvanishing intermediate cohomologies
through the behaviour of system of parameters of the corresponding
module.

In Section~2, we sketch Horrocks-type criteria for ACM and Buchsbaum 
bundles on projective space, especially more surveys on Buchsbaum basics,
not only as an extended introduction, but also
for preparating useful methods such as the Castelnuovo-Mumford regularity
for later sections.

In Section~3, we study a quasi-Buchsbaum vector bundle
on projective space. A vector bundle ${\mathcal E}$ on 
${\mathbb P}^n = {\rm Proj}\, S$ is quasi-Buchsbaum
if ${\displaystyle {\mathfrak m}
{\rm H}^i_* ({\mathcal E}) = 0}$ for $1 \le i \le n-1$ as graded $S$-modules.
There are some structure theorems for quasi-Buchsbaum
bundles of rank 2, say Ellia-Sarti \cite{ES}, Chang \cite{Chang2000},
Kumar-Rao \cite{KR} and Kumar-Peterson-Rao \cite{KPR}
from geometric
interests, see also
Malaspina-Rao \cite{MR2}. Now
we are rather interested in investigating
the algebraic structure of quasi-Buchsbaum modules.
In (\ref{mainth1}) we describe
a quasi-Buchsbaum bundle ${\mathcal E}$ of arbitrary rank 
on ${\mathbb P}^n$ with
${\rm H}^i_*({\mathcal E}) = 0$ for $2 \le i \le n-2$, which
characterizes a null-correlation bundle.

In Section~4, we will take a syzygy theoretic approach for
quasi-Buchsbaum bundles on ${\mathbb P}^3$,
giving somewhat new method to study a free resolution
of vector bundles with simple nonvanishing cohomologies.
Although some results could be worked in general,
we will focus on an experimental study of vector bundles
on ${\mathbb P}^3$ with $\dim_k {\rm H}^1_*({\mathcal E})=
\dim_k{\rm H}^2_*({\mathcal E})=1$, which shows an interesting relationship
from the viewpoint of commutative algebra. Our main result,
Theorem~\ref{mainth}, classifies such vector bundles.
By the behaviour of standard linear systems of parameters, see,
e.g, Trung \cite{T},
we consider a special class of 
of quasi-Buchsbaum bundles as `nonstandard-Buchsbaum',
which characterizes a null-correlation bundle.
We will study the structure of vector
bundles with simple intermediate cohomologies
by using spectral sequences developed in \cite{M1989,MTrieste,M2019}
and constructing free resolutions arising from nonvanishing intermediate
cohomologies.  Moreover, in Theorem~\ref{pseudo-Buchsbaum thm}
and the following remark, the behaviour of system of parameters is shown to
determine the structure of given vector bundles, especially its
hyperplane section property.

In Section~5, what interests us in this context is to classify
a quasi-Buchsbaum bundle ${\mathcal E}$
on ${\mathbb P}^n$
with 
${\rm H}^i_*({\mathcal E}) = 0$ for $1 \le i \le n-1$, $i \ne j_1,j_2$
for given $j_1, j_2$ with
$1 \le j_1 < j_2 \le n-1$.
In particular, in case $\dim_k {\rm H}^1_*({\mathcal E})=
\dim_k{\rm H}^2_*({\mathcal E})=1$, we give several questions and answers
by 
following the flow of the previous section for
$j_2 = j_1 +1$, especially $j_1=1$ and $j_2=2$.

In Section~6, we describe
an application of a syzygy theoretic method to vector bundles
on multiprojective space.

To a large extent, we follow the terminology of algebraic geometry, but where necessary we use concepts from commutative algebra.
The author would like to thank Hajime Kaji, Atsushi Noma and
Ei-ichi Sato for their valuable comments.

\section{Horrocks criterion for ACM and Buchsbaum bundles on projective space}

Let ${\mathcal E}$ be a vector bundle on ${\mathbb P}^n = {\rm Proj}\, S$,
where 
$S = k[x_0,\cdots,x_n]$ and ${\mathfrak m} = (x_0,\cdots,x_n)$.
Let us wirte $M = \Gamma_*({\mathcal E}) = \oplus_{\ell \in {\mathbb Z}}
\Gamma({\mathcal E}(\ell))$ as a graded $S$-module.
 Note that $\dim M = n + 1$, 
${\rm H}^0_{\mathfrak m}(M) = {\rm H}^1_{\mathfrak m}(M) =0$,
${\mathcal E} = \widetilde{M}$ and 
${\rm H}^i_*({\mathcal E}) = {\rm H}^{i+1}_{\mathfrak m}(M)$
for $i \ge 1$. 
Let us start with the splitting of ACM bundles, see, e.g.,
Ottaviani \cite{Ot}.

\begin{df}
\label{ACMdfn}
A vector bundle ${\mathcal E}$ on ${\mathbb P}^n$ is called an
ACM bundle if 
${\displaystyle 
{\rm H}^i_* ({\mathcal E}) = 0}$ for $1 \le i \le n-1$,
\end{df}

\begin{theorem}[Horrocks \cite{Ho}]
\label{Horrocks}
An ACM bundle ${\mathcal E}$ of rank $r$ on ${\mathbb P}^n$ is isomorphic
to a direct sum of line bundles, that is,
${\mathcal E} \cong \oplus_{i=1}^r {\mathcal O}_{{\mathbb P}^n}(\ell_i)$
for some $\ell_1,\cdots,\ell_r \in {\mathbb Z}$.
In other words, a Cohen-Macaulay graded $S$-modlule
is graded free.
\end{theorem}

The first proof, probably best-known, is due to an induction on $n$,  see, e.g., 
\cite[(I, 2.3.1)]{OSS}.  For $n=1$ it is a consequence of Grothendieck theorem.
In case $n \ge 2$, since ${\mathcal E}|_H$ is ACM on $H \cong {\mathbb P}^{n-1}$ 
for a hyperplane $H$ of ${\mathbb P}^n$, we have
${\mathcal E}|_H \cong \oplus {\mathcal O}_H({\ell_i)}$ by the hypothesis
of induction. Let us put ${\mathcal F} = \oplus {\mathcal O}_{{\mathbb P}^n}({\ell_i)}$.
Then we extend $\psi: {\mathcal F}|_H \cong {\mathcal E}|_H$ to a morphism
$\varphi : {\mathcal F} \to  {\mathcal E}$
by the exact sequence 
${\rm Hom}_{{\mathbb P}^n}({\mathcal F}, {\mathcal E})
\to {\rm Hom}_H({\mathcal F}|_H, {\mathcal E}|_H)
\to  {\rm Ext}^1_{{\mathbb P}^n}({\mathcal F}, {\mathcal E}(-1)) = 0$
\noindent 
from the ACM assumption. Since ${\rm det}\, \varphi \in 
\Gamma({\mathbb P}^n, (\wedge^r{\mathcal F})^{\vee} \otimes (\wedge^r{\mathcal E}))
\cong k$ does not vanish on $H$,
we see $\varphi$ is an isomorphism.
\hfill \qed
\medskip

The second proof is 
a consequence of the Auslander-Buchsbaum theorem \cite{AB},
which says that a finitely generated
module $M$ over a Noetherian local ring $R$ with ${\rm proj \, dim}\, M < \infty$
satisfies ${\rm depth}\, M + {\rm proj  \, dim}\, M = {\rm depth}\, R$,
 see, e.g., \cite[(19.1)]{Mat}. 
\hfill \qed
\medskip

The third proof illustrates an interesting application of
basic properties of the Castelnuovo-Mumford regularity.
In fact, the Horrocks theorem immediately follows from Lemma~\ref{ACM}.
\hfill \qed

\begin{dfprop}
\label{CMregularity}
A coherent sheaf ${\mathcal F}$ on ${\mathbb P}^n$ is called $m$-regular 
if ${\rm H}^i({\mathbb P}^n, {\mathcal F}(m-i)) = 0$ for $i \ge 1$.
If ${\mathcal F}$ is $m$-regular, then ${\mathcal F}$ is $(m+1)$-regular
and globally generated, see \cite{Mum,E}. The Castelnuovo-Mumford regularity
${\rm reg}\, {\mathcal F}$ is the minimal integer $m$ such that ${\mathcal F}$
is $m$-regular.
\end{dfprop}

\begin{lem}
\label{ACM}
Let ${\mathcal E}$ a vector bundle on ${\mathbb P}^n$.
Assume that ${\mathcal E}$ is $m$-regular and
${\rm H}^n({\mathcal E}(m-1-n)) \ne 0$ for an integer $m$.
Then ${\mathcal O}_{{\mathbb P}^n}(-m)$ is a direct summand of ${\mathcal E}$.
\end{lem}

\begin{proof}
Since ${\mathcal E}$ is $m$-regular, ${\mathcal E}(m)$ is globally generated.
So we have a surjective map
$\psi: {\mathcal O}_{{\mathbb P}^n}^{\oplus}(-m) \to {\mathcal E}$. 
By Serre duality we have ${\rm H}^0({\mathcal E}^{\vee}(-m)) \ne 0$
from ${\rm H}^n({\mathcal E}(m-1-n)) \ne 0$, which
gives a nonzero map
$\varphi : {\mathcal E} \to {\mathcal O}_{{\mathbb P}^n}(-m)$.
Since $\varphi \circ \psi$ is a nonzero map, 
${\mathcal O}_{{\mathbb P}^n}(-m)$ is a direct summand of ${\mathcal E}$.
\end{proof}

\begin{rem}
\label{Evans-Griffith}
Evans-Griffith \cite{EG} asserts that
if a vector bundle ${\mathcal E}$ on ${\mathbb P}^n$ satisfies
${\rm H}^i_*({\mathcal E}) = 0$ for $1 \le i \le {\rm rank}\, {\mathcal E} -1$,
then ${\mathcal E}$ is a direct sum of line bundle.
See \cite{Ein1} and
\cite[(7.3.1)]{La} for the proof by the Castelnuovo-Mumford regularity
and Le Potier Vanishing Theorem.
\end{rem}

The fourth proof is based on an idea of Horroccks \cite{Ho}. 
The Horrocks correspondence explains the relationship between
the splitting of vector bundles and intermediate cohomologies,
see Walter\cite{Wa}
and Malaspina-Rao\cite{MR}.

\begin{theorem}[Horrocks, Walter, Malaspina-Rao]
\label{HorrocksCorrespondence}
Let $\underline{\mathbb{VB}}$ be the category of vector bundles on 
${\mathbb P}^n = {\rm Proj}\, S$, $S = k[x_0,\cdots.x_n]$. 
modulo stable equivalence. Here vector bundles ${\mathcal E}$ and ${\mathcal F}$ on
${\mathbb P}^n$ are stable equivalent if there are direct sums of line
bundles ${\mathcal L}$ and ${\mathcal M}$ such that
${\mathcal E}\oplus {\mathcal L} \cong {\mathcal F} \oplus {\mathcal M}$.
Let us write ${\mathbb FinL}$ for the full subcategory of 
$C^{\bullet} \in Ob(D^{\flat}($S$-{\mathbb Mod}))$
such that ${\rm H}^i(C^{\bullet})$ is a finite over $S$ and
${\rm H}^i(C^{\bullet}) = 0$, $0 < i < n$. 

A functor $\tau_{>0}\tau_{<n} {\mathbb R}\Gamma_* : \underline{\mathbb{VB}} \to
{\mathbb FinL}$ gives an equivalence of the categories.
\end{theorem}

From the Horrocks correspondence (\ref{HorrocksCorrespondence}),
the vanishing of the intermediate cohomologies 
of a vector bundle  ${\mathcal E}$
on  ${\mathbb P}^n$, that is,
$\tau_{>0}\tau_{<n} {\mathbb R}\Gamma_*({\mathcal E}) = 0$
implies that ${\mathcal E}$ is isomorphic to a direct sum of line bundles,
which gives the Horrocks' theorem.
\hfill \qed \medskip

Next we will explain the Chang-Goto structure theorem
(\ref{Chang-Goto}) of Buchsbaum
bundles on projective space.
We will do the groundwork to extend
their result towards the structure theorem of quasi-Buchsbaum bundles.
Now let us begin with the definition and basic properties
of Buchsbaum rings, see, e.g., Schenzel \cite{Sch} and St\"uckrad-Vogel \cite{SV}.

\begin{dfprop}
\label{BuchsbaumSV}
A graded $S$-module $M$ with $\dim M = d$ is called as a Buchsbaum module
if the following equivalent conditions are satisfied.
\begin{itemize}
\item[{\rm (i)}]
$\ell(M/{\mathfrak q}M) - e({\mathfrak q};M)$ does not depend on
the choice of any homogeneous parameter ideal
${\mathfrak q} = (y_1,\cdots,y_d)$.
\item[{\rm (ii)}]
For any homogeneous system $y_1,\cdots,y_d$ of parameters \par
$ {\mathfrak m} {\rm H}^j_{\mathfrak m}(M/(y_1,\cdots,y_i)M) = 0$ for
$0 \le i \le d-1, \, 0 \le j \le d-i-1$.
\item[{\rm (iii)}]
$\tau_{<d}{\mathbb R}\Gamma_{\mathfrak m}(M)$ is isomorphic to a complex of
$k$-vector spaces in $D^{\flat}($S$-{\mathbb Mod})$.
\end{itemize}
\end{dfprop}

\begin{df}
Let $S=k[x_0,\cdots,x_n]$ be the polynomial ring over a field $k$ with
${\mathfrak m} = (x_0,\cdots,x_n)$.
A vector bundle ${\mathcal E}$ on
${\mathbb P}^n = {\rm Proj}\, S$ is called a Buchsbaum bundle if 
${\displaystyle {\mathfrak m}{\rm H}^i_* ({\mathcal E}|_L)
 = 0, \, 1 \le i \le r-1}$
for any $r$-plane $L (\subseteq {\mathbb P}^n)$, $r=1,\cdots,n$.
\end{df}

\begin{rem}
A vector bundle ${\mathcal E}$ on
${\mathbb P}^n = {\rm Proj}\, S$ is Buchsbaum
if and only of a graded $S$-module $M = \Gamma_*({\mathcal E})$
is Buchsbaum.
\end{rem}

The Koszul complex $K_{\bullet} = K_{\bullet}((x_0,\cdots,x_n);S)$
gives the minimal free resolution of a graded $S$-module $k= S/{\mathfrak m}$. Then
$\Omega^{p-1}_{{\mathbb P}^n} = \widetilde{E_p}$, where $E_p$ is the
$p$-th syzygy of a graded $S$-module $k$. Note that $E_0 \cong k$,
$E_1 \cong {\mathfrak m}$ and $E_{n+1} \cong S(-n-1)$.

\begin{theorem}[Chang \cite{C}, Goto \cite{G}]
\label{Chang-Goto}
A Buchsbaum bundle ${\mathcal E}$ on ${\mathbb P}^n$ is isomorphic
to a direct sum of sheaves of differential form, that is,
${\mathcal E} \cong \oplus_{i=1}^r \Omega_{{\mathbb P}^n}^{k_i}(\ell_i)$.

In other words, a graded $S$-module $M = \Gamma_*({\mathcal E})$
is isomorphic to $\oplus E_{p_i}(\ell_i)$.
\end{theorem}

\begin{ex}
\label{BuchsbaumCurve}
The structure of ACM and Buchsbaum modules over
the polynomial ring is given by (\ref{Chang-Goto}).
Let us describe explicitly the structure of geometric examples such as
rational curves $X$ and $Y$ of degree 3 and 4 in 
${\mathbb P}^3$. Then (i) $X$ is ACM and (ii) $Y$ is Buchsbaum
with ${\rm H}^1_*({\mathcal I}_{Y/{\mathbb P}^3}) \cong k(-4)$.
\begin{enumerate}
\item[{\rm (i)}]
Let us consider
a morphism $\varphi : X \to {\mathbb P}^1$ of rank 3
by $S=k[s^3,t^3] \subset k[s^3,s^2t,st^3,t^3]$,
Then we have an ACM module $M = S/\Gamma_*({\mathcal I}_X)$ 
over the polynomial ring $S$, so $M$ is graded free as
$M = S \cdot 1 \oplus S \cdot s^2t
\oplus S \cdot st^2$. 
\item[{\rm (ii)}]
Let us consider
a morphism $\varphi : Y \to {\mathbb P}^1$ of rank 4
by $T=k[s^4,t^4] \subset k[s^4,s^3t,st^3,t^4]$.
Then we have a (not ACM) Buchsbaum module $N = T/\Gamma_*({\mathcal I}_Y)$ 
over the polynomial ring $T$, so a (not free) Buchsbaum module
 $N$ is written as
$N = T \cdot 1 \oplus T \cdot s^3t
\oplus T \cdot st^3 \oplus L$, where
$L = (s^4,t^4)T$ is the first syzygy module of $T$. 
\end{enumerate}
\end{ex}

Now we will sketch several proofs of (\ref{Chang-Goto}).
The first proof \cite{G} is based 
on `surjectivity criterion' for Buchsbaum
modules and technical lemma (\ref{keylem}).
The second proof \cite{C} studies a map from $\Omega^p_{{\mathbb P}^n}(p)$ to
${\mathcal O}_{{\mathbb P}^n}$ in Case B, which has driven us to research
in later sections.
Both  are due to descending induction by ${\rm depth}\, M$ and
${\rm depth}\, M^{\vee}$.
\medskip

The first proof
uses descending induction on $t  = \depth_S M \ge 2$. 
It is clear for the Cohen-Macaulay case, that is, $t = n+1$. Let us assume that $t \le n$.
From a presentation $0 \to N \stackrel{f}{\to} F \stackrel{g}{\to} M \to 0$, where $F$ is
graded free, we see  $N$ is Buchsbaum and $\depth_S N = t + 1$. Then we have
$N \cong \oplus E_{p_i}(\ell_i)$.
By the dual sequence
$0 \to M^{\vee} \to F^{\vee} \to N^{\vee} \stackrel{\partial}{\to} {\rm Ext}_S^1(M,S) \to 0$,
we have short exact sequences
\[ 0 \to M^{\vee} \to F^{\vee} \to W \to 0 \quad {\rm and}
\quad 0 \to W \to N^{\vee} \stackrel{\partial}{\to} {\rm Ext}_S^1(M,S) \to 0. \]
Now we will prove that
$W$ is isomorphic  to a direct sum of some copies of $E_p(\ell)$'s.
Then $\widetilde{M^{\vee}}$ is isomorphic to a direct sum
of sheaves of differential $p$-form with some twist by Lemma~\ref{keylem}, 
and so is ${\mathcal E}$
as desired.

In order to describe the structure of $W$, let us put $N = N' \oplus N''$, where 
$N' \cong \oplus_{{t+1} \le k_i \le n} E_{p_i}(\ell_i)$ and $N'' \cong \oplus E_{n+1}(\ell'_j)$.
Note that ${\rm Ext}_S^1(M,S)$ is a $k$-vector space, that is,
${\mathfrak m}{\rm Ext}_S^1(M,S) = 0$ from the Buchsbaum property. So,
we have only to show $\partial(N'^{\vee}) = 0$, that is, suffice to prove
$\partial (E_j) = 0$ for $j=1,\cdots, n$,
and by local duality, equivalently, ${\rm H}^n_{\mathfrak m}(M) \to {\rm H}^{n+1}_{\mathfrak m}(N)
\to {\rm H}^{n+1}_{\mathfrak m}(N')$ is zero.
From the commutative diagram with exact rows
\[
\begin{array}{ccccc}
  {\rm H}^n(x_0,\cdots,x_n; M) & \to &  {\rm H}^{n+1}(x_0,\cdots,x_n; N)
 & \to &   {\rm H}^{n+1}(x_0,\cdots,x_n; N')  \\
\downarrow  & & \downarrow & & \downarrow  \\
  {\rm H}^n_{\mathfrak m}(M) & \to & {\rm H}^{n+1}_{\mathfrak m}(N)
& \to  & {\rm H}^{n+1}_{\mathfrak m}(N')    \\
\end{array}
\]
and the surjectivity of the left downarrow from the Buchsbaum property, what we need to show
is ${\rm H}^{n+1}(x_0,\cdots,x_n; N') \to  {\rm H}^{n+1}_{\mathfrak m}(N')$ is zero, which
follows from Remark~\ref{gotolem}.
\hfill \qed

\begin{rem}
\label{gotolem}
Let $S=k[x_0,\cdots,x_n]$ be the polynomial ring.
Let $E_j$ be the $j$-th syzygy module. Then
the natural map ${\rm H}^{n+1}(x_0,\cdots,x_n; E_j) \to  {\rm H}^{n+1}_{\mathfrak m}(E_j)$ is zero
for $1 \le j \le n$.
Indeed,
we give another proof of \cite[(2.11)]{G}.
A graded $S$-module
${\rm H}^{n+1}(x_0,\cdots,x_n; E_j) (\cong {\rm Ext}^{n+1}_S(k,E_p)) \cong (E_p/{\mathfrak m}E_p)(n+1)$
has nonzero elements only in degree $-n+p$. On the other hand, ${\rm H}_{\mathfrak m}^{n+1}(E_p)
\cong {\rm Hom}_k(E_{n-p},k)$ has nonzero elements only in degree $\le -n+p-1$. 
Thus the assertion is proved.  \hfill \qed
\end{rem}

\begin{rem}
\label{goto}
The structure theorem also works for $t = {\rm depth}_SM \le 1$, see \cite{G}.
Indeed, in case $t = 0$, it can be reduced to the case $t \ge 1$ because
${\rm H}^0_{\mathfrak m}(M) \cap {\mathfrak m}M = (0)$ implies
that ${\rm H}^0_{\mathfrak m}(M)$ is a direct summand of $M$.
In case $t = 1$, an exact sequence $ 0 \to M \to M^{\vee \vee} \to
{\rm Ext}^1_S(W,S) \to 0$ gives the assertion as in the first proof of (\ref{Chang-Goto}).
\end{rem}

For the second proof, 
let us denote $i({\mathcal E}) (= n + 1 - {\rm depth}\, M^{\vee})$
as the maximal integer $i$ such that ${\rm H}^p_*({\mathcal E}) = 0$,
$i + 1 \le p \le n-1$. We use induction on $i = i({\mathcal E})$. It is clear for $i = 0$.
Let us assume $i \ge 1$. 
If ${\rm H}^1_*({\mathcal E}) \ne 0$, we have a short exact sequence
of vector bundles
$0 \to {\mathcal E} \to {\mathcal F} \to {\mathcal M} \to 0$, where
${\mathcal F}$ is a vector bundle
with ${\rm H}^1_*({\mathcal F}) = 0$ and ${\mathcal M}$ is a direct sum of line
bundles on ${\mathbb P}^n$
by (\ref{lem1}).

The minimal generator of $\Gamma_*({\mathcal F}^{\vee})$
 give a short exact sequence
$0 \to {\mathcal F} \to {\mathcal L} \to {\mathcal K} \to 0$,
 where ${\mathcal L}$ is a sum
of line bundles. Then ${\mathcal K}$ is Buchsbaum with 
$i({\mathcal K}) = i({\mathcal E}) - 1$.
Thus we have ${\mathcal K}$ is isomorphic to a direct sum of $\Omega^{p_j}_{{\mathbb P}^n}(k_j)$'s, and
so is ${\mathcal F}$ by Lemma~\ref{keylem}.

Note that ${\rm Hom}({\Omega}^p_{{\mathbb P}^n}(\ell), {\mathcal O}_{{\mathbb P}^n}) \ne 0$ if and only if
$\ell \le p$. 
We may assume  $k_j \le p_j$ if $p_j \ge 2$. and $k_j < 0$
 if $p_j = 0$,
where
${\mathcal F} = \oplus \Omega_{{\mathbb P}^n}^{p_j}(k_j)$ and ${\mathcal L} = \oplus
{\mathcal O}_{{\mathbb P}^n}(-c_i)^{\oplus \gamma_i}$, $0=c_1<\cdots<c_s$.
\medskip

\noindent
{\bf Case A}.
Let us consider the case $k_j < p_j$ for all $p_j \ge 0$. 
In order to prove the assertion,
the exact sequence
$0 \to {\mathcal E} \to {\mathcal F} \to {\mathcal L} \to 0$
shown to have
has a sequence $0 \to \Omega^1_{{\mathbb P}^n} \to
{\mathcal O}_{{\mathbb P}^n}^{n+1} \to {\mathcal O}_{{\mathbb P}^n} \to 0$
as a direct summand by using only the quasi-Buchsbaum property,
${\mathfrak m}{\rm H}^i_*({\mathcal E}) =0$
for $1 \le i \le n-1$, see \cite[page 330 Case 1]{C} for the details.
\medskip

\noindent
{\bf Case B}.
Let us consider the case $k_j = p_j$ for some $j$, that is, ${\mathcal F}$ has a direct summand of the form
$\Omega_{{\mathbb P}^n}^q(q)$ for some $q>1$. 
In order to prove the assertion,
$\Omega_{{\mathbb P}^n}^q(q) \to {\mathcal O}_{{\mathbb P}^n}$ is shown to be
zero in the map ${\mathcal F} \to {\mathcal L}$ by using
the Buchsbaum property of ${\mathcal E}$, see 
\cite[page 331 Case 2]{C} for the details.
\hfill \qed

\begin{lem}
\label{lem1}
Let ${\mathcal E}$ be a vector bundle on ${\mathbb P}^n$ with
${\rm H}^1_*({\mathcal E}) \ne 0$. Then
there exists a short exact sequence of vector bundles
$0 \to {\mathcal E} \to {\mathcal F} \to {\mathcal L} \to 0$, where
${\mathcal F}$ is a vector bundle
with ${\rm H}^1_*({\mathcal F}) = 0$ and ${\mathcal L}$ is a direct sum of line
bundles on ${\mathbb P}^n$.
\end{lem}
%Indeed, a nonzero element of
%${\rm H}^1({\mathcal E}(-\ell))$ gives a short
%exact sequences
%$ 0 \to {\mathcal E}_1 \to {\mathcal E} \to {\mathcal O}_{{\mathbb P}^n}(\ell) \to 0$.
%Then we repeat this process.

\begin{proof}
Let  $s = \dim  {\rm H}^1_*({\mathcal E}) > 0$.
A nonzero element of
${\rm H}^1({\mathcal E}(-\ell_1))$ gives a short
exact sequences
$ 0 \to {\mathcal E}_1 \to {\mathcal E} \to {\mathcal O}_{{\mathbb P}^n}(\ell_1) \to 0$.
Then $\dim {\rm H}^1_*({\mathcal E}_1) = \dim {\rm H}^1_*({\mathcal E}) -1$.
By repeating this process, we have vector bundles ${\mathcal E}_i$ satisfying
short exact sequences $0 \to {\mathcal E}_i \to {\mathcal E}_{i+1} \to
{\mathcal O}_{{\mathbb P}^n}(\ell_{i+1}) \to 0$, where ${\mathcal E}_0 = {\mathcal E}$ and
$\dim {\rm H}^1_*({\mathcal E}_{i+1}) = \dim {\rm H}^1_*({\mathcal E}_i) -1$ for $i=0,\cdots,s-1$.
Since the exact sequence $0 \to {\mathcal E}_i/{\mathcal E} \to {\mathcal E}_{i+1}/{\mathcal E} \to
{\mathcal O}_{{\mathbb P}^n}(\ell_{i+1}) \to 0$ is inductively shown to split, we see ${\mathcal F}/{\mathcal E}
\cong \oplus_{i=1}^s {\mathcal O}_{{\mathbb P}^n}(\ell_i)$ by taking ${\mathcal F}={\mathcal E}_i$, as desired.
\end{proof}

\begin{lem}{\rm (cf. \cite[(1.3)]{C}, \cite[(3.5.2)]{G})}
\label{keylem}
Let ${\mathcal E}$ be a vector bundle on ${\mathbb P}^n$ with
${\rm H}^1_*({\mathcal E}) = 0$.
Assume that there is an exact sequence $ 0 \to {\mathcal E} \to {\mathcal L} \to {\mathcal F} \to 0$,
where ${\mathcal L}$ is a direct sum of line bundles not being any summand of ${\mathcal E}$,
 and ${\mathcal F} = \oplus_{p_j \ge 1} \Omega_{{\mathbb P}^n}^{p_j}(k_j)$.
Then we have ${\mathcal E} \cong \oplus_{p_j \ge 1} \Omega_{{\mathbb P}^n}^{p_j+1}(k_j)$.
\end{lem}

\begin{proof} %[Sketch of Proof]
We may assume 
${\mathcal F} =  {\mathcal F}' \oplus (\oplus_{q \ge 1} \Omega_{{\mathbb P}^n}^q(q+1)^{\oplus})$,
where ${\mathcal F}' =  \oplus_{p_j \le k_j} \Omega^{p_j}(k_j)$, 
and ${\mathcal L}$ has no direct summand of positive degree.
Since ${\rm H}^1_*({\mathcal E}) = 0$ and a global section $\Omega_{{\mathbb P}^n}^q(q+1)$
is lifted up to a section of ${\mathcal L}$, 
there is a direct summand ${\mathcal O}_{{\mathbb P}^n}^N$ of ${\mathcal L}$.
Then an exact sequence $0 \to \oplus_{q \ge 1} \Omega_{{\mathbb P}^n}^{q+1}(q+1)^{\oplus}
\to {\mathcal O}_{{\mathbb P}^n}^N \to \oplus_{q \ge 1} \Omega_{{\mathbb P}^n}^q(q+1)^{\oplus} \to 0$
gives a direct summand of the sequence
$ 0 \to {\mathcal E} \to {\mathcal L} \to {\mathcal F} \to 0$,
which gives the assertion by repeating this process.
\end{proof}

The third proof pointed out by Yoshino \cite{Y} 
depends on the Horrocks correspondence. 
Indeed,
by  (\ref{BuchsbaumSV}) (iii),
$\tau_{>0}\tau_{<n} {\mathbb R}\Gamma_*({\mathcal E})
(\cong \tau_{<n+1}{\mathbb R}\Gamma_{\mathfrak m}(M))$
is isomorphic to a compex of $k$-vector spaces,
where $M = \Gamma_*({\mathcal E})$
is a graded $S$-module.
Since
$\tau_{>0}\tau_{<n} {\mathbb R}\Gamma_*( \Omega^p_{{\mathbb P}^n})$
is isomorphic to a complex of one $k$-vector space, we see that 
a Buchsbaum bundle
${\mathcal E}$ is isomorphic to a direct sum of the sheaves of differential form
$\Omega^p_{{\mathbb P}^n}(\ell)$ under stable equivalence
from the categorical equivalence
(\ref{HorrocksCorrespondence}).
\hfill \qed \medskip

Finally,
as the fourth proof, we explain original one coming from an idea of 
\cite{BM,MM}, which becomes the starting point of a
syzygy theoretic method.
First we will give a summary of the Buchsbaum criterion in terms
of spectral sequence \cite{M1989,MTrieste, M2019} in order
to apply to the syzgy theoretic proof of the structure theorem.
For a graded $S$-module $M = \Gamma_*({\mathcal E})$,
we consider
a Koszul complex $K_{\bullet} = K_{\bullet}((x_0,\cdots,x_n);S)$
and a \u{C}ech complex  $L^{\bullet} = (0 \to M \to
C^{\bullet}({\cal U}; {\mathcal E})[-1])$, where ${\mathcal U}$ is
an affine covering of ${\mathbb P}^n$.
Then we take
a double complex
${\rm Hom}_S(K_{\bullet},L^{\bullet})$, which yields
a spectral sequence $\{ {\rm E}^{i,j}_r \}$ such that
\[ {\rm E}^{i,j}_1 = 
{\rm H}_i((x_0,\cdots,x_n); {\rm H}_{\mathfrak m}^j(M))
 \Rightarrow {\rm H}^{i+j} = {\rm H}^{i+j}((x_0,\cdots,x_n);M). \]
The natural map
${\displaystyle {\rm H}^j = {\rm H}^j((x_0,\cdots,x_n);M) \to
{\rm E}^{0,j}_1 = {\rm H}_{\mathfrak m}^j(M)}$
is surjective for $0 \le j \le n$
from the theory of Buchsbaum ring (cf.\cite{MTrieste,SV}),
and
$d_r^{i,j} : {\rm E}_r^{i,j} \to {\rm E}_r^{i+r,j-r+1}$ is zero for
$j \le n$, $r \ge 1$ (\cite[(1.11)]{M1989}).

Keeping the construction above in mind,
we will give an analog of the
spectral sequence. From the Koszul complex $K_{\bullet}
= K_{\bullet}((x_0,\cdots,x_n);S)$,
we have exact sequences:
\begin{itemize}
\item[] $\bar{N}^{\bullet} \, : \,$
$0 \to \Omega_{{\mathbb P}^n}^{p}  \to
 {\mathcal O}_{{\mathbb P}^n}^{\oplus a_p}(-p)
\to \cdots \to  {\mathcal O}_{{\mathbb P}^n}^{\oplus a_1}(-1)  \to
 {\mathcal O}_{{\mathbb P}^n} \to 0$
\vspace{-0.2cm}
\item[] $\bar{\bar{N}}^{\bullet} \, : \,$
$ 0 \to  {\mathcal O}_{{\mathbb P}^n}(-n-1) \to
 {\mathcal O}_{{\mathbb P}^n}^{\oplus a_n}(-n)
\to \cdots \to  {\mathcal O}_{{\mathbb P}^n}^{\oplus a_{p+1}}(-p-1)  \to
 \Omega_{{\mathbb P}^n}^{p} \to 0$,
\end{itemize}

\noindent
where 
$\bar{N}^{-i} = {\mathcal O}_{{\mathbb P}^n}^{\oplus a_i}(-i)$ for $0 \le i \le p$,
$\bar{N}^{-p-1} =  \Omega_{{\mathbb P}^n}^p$
and $\bar{\bar{N}}^{-i} = {\mathcal O}_{{\mathbb P}^n}^{\oplus a_i}(-i)$
for $p+1 \le i \le n+1$,
$\bar{\bar{N}}^{-p} = \Omega_{{\mathbb P}^n}^p$,
${\displaystyle a_r = \left(
\begin{array}{c}
n+1 \\ r
\end{array}
\right) }$.
Then the exact sequences
\begin{itemize}
\item[]
$0 \to {\mathcal E} \to
 {\mathcal E}^{\oplus a_1}(1)
\to \cdots \to  {\mathcal E}^{\oplus a_p}(p) \to
{\mathcal E} \otimes \Omega_{{\mathbb P}^n}^{p \vee} \to 0$
\item[]
$0 \to {\mathcal E}^{\vee}(-n-1) \to
 {{\mathcal E}^{\vee}}^{\oplus a_n}(-n)
\to \cdots \to  {{\mathcal E}^{\vee}}^{\oplus a_{p+1}}(-p-1) \to
{\mathcal E}^{\vee} \otimes \Omega_{{\mathbb P}^n}^p \to 0$
\end{itemize}
give maps
$\, \varphi: {\rm H}^0({\mathcal E} \otimes
\Omega^{p \vee}_{{\mathbb P}^n}) \to {\rm H}^p({\mathcal E}) \quad {\rm and} \quad
\psi : {\rm H}^0({\mathcal E}^{\vee} \otimes
\Omega^p_{{\mathbb P}^n}) \to {\rm H}^{n-p}({\mathcal E}^{\vee}(-n-1))$.
\smallskip

\begin{lem}[cf.~\cite{MM}]
\label{BuchsbaumLemma}
Under the condition above, assume that there is a nonzero element
$s \in {\rm H}^p({\mathcal E})$ and a corresponding element $t
\in {\rm H}^{n-p}({\mathcal E}^{\vee}(-n-1))$ by Serre duality
satisfying that $s$ and $t$ can be lifted up to ${\rm H}^0({\mathcal E} \otimes
\Omega^{p \vee}_{{\mathbb P}^n})$ and ${\rm H}^0({\mathcal E}^{\vee} \otimes
\Omega^p_{{\mathbb P}^n})$ by $\varphi$ ans $\psi$ respectively.
Then $\Omega^p_{{\mathbb P}^n}$ is a direct summand of ${\mathcal E}$.
\end{lem}

\begin{proof}
For $s(\ne 0) \in {\rm H}^p({\mathcal E})$ there exists
$f \in {\rm H}^0({\mathcal E} \otimes \Omega^{p \vee}_{{\mathbb P}^n})$
such that $\varphi(f) = s (\ne 0) \in {\rm H}^p({\mathcal E})$
Let us take $s \in {\rm H}^{m}({\mathcal E})$
and  $g \in
{\rm H}^0({\mathcal E}^{\vee} \otimes
\Omega^p_{{\mathbb P}^n})$
corresponding to
 $t \in {\rm H}^{n-p}({\mathcal E}^{\vee}
(-n-1))$ and
$\psi(g) = t (\ne 0) \in
{\rm H}^{n-p}({\mathcal E}^{\vee}(-n-1))$
respectively. Then we regard
$f$ and $g$ as elements of ${\rm Hom}(\Omega^p_{{\mathbb P}^n}, {\mathcal E})$
and ${\rm Hom}({\mathcal E},\Omega^p_{{\mathbb P}^n})$ respectively.
From a commutative diagram
\[
\begin{array}{ccccc}
{\rm H}^0({\mathcal E} \otimes \Omega^{p \vee}_{{\mathbb P}^n})
\otimes {\rm H}^0({\mathcal E}^{\vee} \otimes \Omega^p_{{\mathbb P}^n})
& \to  &
{\rm H}^0(\Omega^{p \vee}_{{\mathbb P}^n} \otimes \Omega^p_{{\mathbb P}^n}) & \cong &
 {\rm H}^{0}({\cal O}_{{\mathbb P}^n})
\\
\downarrow & & & & \downarrow   \\
{\rm H}^p({\mathcal E}) \otimes {\rm H}^{n-p}({\mathcal E}^{\vee}(-n-1))
  & 
\to  & & &
 {\rm H}^n({\cal O}_{{\mathbb P}^n}(-n-1) ), 
\\
\end{array}
\]
\noindent
the natural map
${\displaystyle {\rm H}^0({\mathcal E} \otimes \Omega^{p \vee}_{{\mathbb P}^n})
\otimes {\rm H}^0({\mathcal E}^{\vee} \otimes \Omega^p_{{\mathbb P}^n})
 \to
 {\rm H}^{0}({\cal O}_{{\mathbb P}^n})}$ gives an isomorphism
$g \circ f$. Thus we obtain that
 $\Omega^p_{{\mathbb P}^n}$ is a direct summand of ${\mathcal E}$.
\end{proof}

Let us come back to the proof of (\ref{Chang-Goto})
and consider a double complex ${\rm Hom}_S(\bar{N}^{\bullet},L^{\bullet})$.
From the spectral sequence theory of Buchsbaum modules,
we have a spectral sequence $\{ {\rm F}^{i,j}_r \}$ which satisfies that
${\rm F}^{i,j}_1 = {\rm H}^{j-1}_*({\mathcal E}(i))^{\oplus a_i}$ for $0 \le i \le p$
and $j \ge 2$
and ${\rm F}^{p+1,j}_1 = {\rm H}^{j-1}_*({\mathcal E} \otimes
\Omega_{{\mathbb P}^n}^{p \vee})$. Note that the maps $d^{i,j}_r : {\rm F}^{i,j}_r
\to {\rm F}^{i+r.j-r+1}_r$ is zero for $i, j, r$ with $1 \le r \le j \le n-1$ and 
$0 \le i \le p-r$ from the comparison between $\{ {\rm E}^{i,j}_r \}$
and $\{ {\rm F}^{i,j}_r \}$. Thus we have a surjective map
$\varphi: {\rm H}^0({\mathcal E} \otimes
\Omega^{p \vee}_{{\mathbb P}^n}) \to {\rm H}^p({\mathcal E})$.
Similarly, a double complex constructed from
the \u{C}ech resolution of ${\mathcal E}^{\vee} \otimes
\bar{\bar{N}}^{\bullet}$ also gives
a surjective map $\psi : {\rm H}^0({\mathcal E}^{\vee} \otimes
\Omega^p_{{\mathbb P}^n}) \to {\rm H}^{n-p}({\mathcal E}^{\vee}(-n-1))$
from the spectral sequence criterion of Buchsbaum modules. Hence we have
the assertion by (\ref{BuchsbaumLemma}). 
\hfill \qed

\section{Towards Structure Theorem of quasi-Buchsbaum bundles on projective space}
%--- Null-Correlation Bundles}

In this section we study quasi-Buchsbaum
vector bundles and gives a characterization of null-correlation bundles on 
${\mathbb P}^n$.
Our result (\ref{mainth1}) gives a classification of indecomposable vector
bundles 
${\mathcal E}$
on ${\mathbb P}^n$ with $\dim_k {\rm H}^1_*({\mathcal E}) 
= \dim_k {\rm H}^{n-1}_*({\mathcal E}) = 1$, which characterizes
a null-correlation bundle.
Throughout this section we assume ${\rm char}\, k \ne 2$
to simplify a standardization of skew-symmetric matirices.
\medskip

%The proof is based on (\ref{lem1}) in the second proof and 
%the regularity technique (\ref{observation}).

%On the other hand,
%later in Section~5, we will take a syzygy theoretic way to this topic.
%We wish to extend this criterion to characterize
%some interesting examples of algebraic vector bundles such as
%the Horrocks-Mumford
%bundle.
%\medskip

Let us define a null-correlation bundle, see, e.g., \cite{OSS}.
%and also see \cite{Ein} for generalized null-correlation bundles.
Let ${\mathbb P}^n = \Proj S$, $S = k[x_0,\cdots,x_n]$,
${\mathfrak m} = (x_0,\cdots,x_n)$ with $n$ odd. 
Let us write an element of $\Gamma(\Omega_{{\mathbb P}^n}(2))$ explicitly.
Since $\Gamma(\Omega_{{\mathbb P}^n}(2))$ is the kernel
of $\Gamma({{\mathcal O}_{{\mathbb P}^n}}(1))^{\oplus n+1} \to
\Gamma({{\mathcal O}_{{\mathbb P}^n}}(2))$ in the Euler sequence,
we have an element
$(a_{00}x_0 + \cdots + a_{0n}x_n, \cdots, a_{n0}x_0 + \cdots + a_{nn}x_n)$
of $\Gamma({{\mathcal O}_{{\mathbb P}^n}}(1))^{\oplus n+1}$
satisfying that $\sum_{ij}a_{ij}x_ix_j = 0$, where $a_{ij} \in k$, $i, j = 0, \cdots, n$.
Then we have a skew-symmetric $(n+1) \times (n+1)$-matrix $A = (a_{ij})$, which
gives a map ${\mathcal O}_{{\mathbb P}^n} \to \Omega_{{\mathbb P}^n}(2)$.
Now assume that ${\rm rank}\, A = n+1$. Then the cokernel of this map
defines a vector bundle of rank $n-1$.
By standardizing the skew-symmetric matrix, we may take
${\displaystyle A =
\left(
\begin{array}{cc}
O & B \\
-B & O \\
\end{array}
\right), }$
where $B$ is a diagonal matrix ${\rm diag}\, (\lambda_1,\cdots,\lambda_{(n+1)/2})$ with $\lambda_i (\ne 0) \in k$.
By taking the dual and twisting of
 ${\mathcal O}_{{\mathbb P}^n} \to \Omega_{{\mathbb P}^n}(2)$,
we have a surjective morphism
$\varphi: {\mathcal T}_{{\mathbb P}^n}(-1) \to {\mathcal O}_{{\mathbb P}^n}(1)$.
Then a null-correlation bundle ${\mathcal N}$ is defined as $\Ker \varphi$, 
which gives a short exact sequence
\[ 0 \to {\mathcal N} \to
 {\mathcal T}_{{\mathbb P}^n}(-1)(\cong \Omega^{n-1}_{{\mathbb P}^n}(n)) \to
 {\mathcal O}_{{\mathbb P}^n}(1) \to 0. \]
Thus we see that a skew-symmetric $(n+1) \times (n+1)$-matrix $A$ of
rank $n+1$ defines a null-correlation bundle ${\mathcal N}$, which does not
depend on a choice of $A$ via a coordinate change. 
For example, $(x_1,-x_0,x_3,-x_2,\cdots,x_n,-x_{n-1}) 
\in \Gamma(\Omega_{{\mathbb P}^n}(2))$ gives a null-correlation bundle through a
morphism
${\mathcal O}_{{\mathbb P}^n} \to \Omega_{{\mathbb P}^n}(2)$.
\medskip

The intermediate cohomologies appear only in ${\rm H}^1({\mathcal N}(-1)) (\cong k)$
and ${\rm H}^{n-1}({\mathcal N}(-n)) (\cong k)$.
Also, note that ${\rm H}^0({\mathcal N}(\ell)) = 0$ for $\ell \le 0$
and ${\rm H}^n({\mathcal N}(\ell)) = 0$ for $\ell \ge -n-1$.
The dual ${\mathcal N}^{\vee}$ has the same cohomology table as ${\mathcal N}$,
and ${\mathcal N}^{\vee} \cong \wedge^{n-2} {\mathcal N}$, especially,
${\mathcal N}$ is self-dual for $n=3$. 
Thus we see that ${\rm reg}\, {\mathcal N} = 1$ and ${\rm reg}\, {\mathcal N}^{\vee} = 1$.

\begin{df}
A vector bundle ${\mathcal E}$ on
${\mathbb P}^n$ is called a quasi-Buchsbaum bundle if 
${\displaystyle {\mathfrak m}{\rm H}^i_* ({\mathcal E})
 = 0, \, 1 \le i \le n-1}$.
\end{df}

\begin{rem}
A Buchsbaum vector bundle is a quasi-Buchsbaum vector bundle.
A null-correlation bundle is quasi-Buchsbaum but not Buchsbaum.
\end{rem}

\begin{ques}
Is there a structure theorem of quasi-Buchsbaum bundle on the projective space?
Can we classify quasi-Buchsbaum bundles ${\mathcal E}$ on ${\mathbb P}^n$ with
${\rm H}^i_*({\mathcal E}) = 0$, $2 \le i \le n-2$?

\end{ques}
There is an answer to this question for stable vector bundles of rank 2 on
${\mathbb P}^3_{\mathbb C}$ by using Barth's Restriction Theorem,
see \cite{ES,KR}.

\begin{prop}[Ellia-Sarti(1999), Kumar-Rao(2000)]
\label{ES-KR}
Let ${\mathcal E}$ be a stable (resp.~indecomposable)
vector bundle of rank 2 on ${\mathbb P}^3_{\mathbb C}$.
Then ${\mathcal E}$ is quasi-Buchsbaum if and only if ${\mathcal E}$ is
a null-correlation bundle with some twist.
\end{prop}

\begin{rem}
\label{skew-symmetric}
Let us take a map $\varphi : {\mathcal O}_{{\mathbb P}^3} \to \Omega_{{\mathbb P}^3}(2)$
corresponding to a skew-symmetric $4 \times 4$-matrix $A = (a_{ij})$.
Then we have
%\vspace{0.2cm}
%
%\noindent
${\displaystyle  P^{-1}AP = \left(
\begin{array}{cccc}
0  & \lambda & 0 & 0 \\
-\lambda & 0 & 0 & 0 \\
0 & 0 & 0 & \mu \\
0 & 0 & -\mu & 0
\end{array}
\right)    }$
for an invertible matrix $P$.
So, there are 3 types as the cokernel of $\varphi$
\vspace{0.2cm}

\noindent
(i)
${\displaystyle  A = \left(
\begin{array}{cccc}
0  & 1 & 0 & 0 \\
-1 & 0 & 0 & 0 \\
0 & 0 & 0 & 1 \\
0 & 0 & -1 & 0
\end{array}
\right)  \quad  }$
(ii)
${\displaystyle  A = \left(
\begin{array}{cccc}
0  & 1 & 0 & 0 \\
-1 & 0 & 0 & 0 \\
0 & 0 & 0 & 0 \\
0 & 0 & 0 & 0
\end{array}
\right)  \quad   }$
(iii) $A = 0$
\vspace{0.2cm}

\noindent
A null-correlation bundle ${\mathcal N}$ is defined as
${\mathcal N}^{\vee}(1) \cong {\rm Coker}\, \varphi$ in case (i).
In general, a skew-symmetric matrix $(n+1) \times (n+1)$-
matrix $A = (a_{ij})$ corresponding to
 $\varphi : {\mathcal O}_{{\mathbb P}^n} \to \Omega_{{\mathbb P}^n}(2)$.
has an even number rank.
In case ${\rm rank}\, A = 2m$, there is an invertible matrix $P$ such that
${\displaystyle P^{-1}AP =
\left(
\begin{array}{ccc}
0  & B & 0 \\
-B & 0 & 0 \\
0 & 0 & 0 
\end{array}
\right)}$ with a $m$-diagonal matrix $B = {\rm diag}\, (\lambda_1,\cdots,\lambda_m)$, $\lambda_i \ne 0$.
\hfill \qed
\end{rem}

\begin{df}
Let ${\mathcal F}$ be a coherent sheaf on ${\mathbb P}^n$.
Then we define
\[ a_i({\mathcal F}) = \max \{ \ell \in {\mathbb Z} |
{\rm H}^i({\mathcal F}(\ell)) \ne 0 \} \]
for $i \ge 1$, possibly $a_i({\mathcal F}) = - \infty$.
Note that ${\rm reg}\, {\mathcal F} = \max \{ a_i({\cal F}) + i + 1 | i \ge 1 \}$.
\end{df}

\begin{lem}
\label{observation}
Let $p$ and $q$ be integers with $1 \le p < q \le n-1$.
Let ${\mathcal E}$ be vector
bundle on ${\mathbb P}^n$
such that ${\rm H}^i_*({\mathcal E})=0$ for $1 \le i \le n-1$ with $i \ne p,q$. 
Then there is a quasi-Buchsbaum bundle ${\mathcal F}$ with
$a_p({\mathcal F}) = -p$, $a_q({\mathcal F}) = - q -1$ and
$a_n({\mathcal F}) \le -n-1$ such that
${\mathcal E}(\ell) \cong {\mathcal F} \oplus (\oplus_i {\mathcal O}_{{\mathbb P}^n}(\ell'_i))
\oplus \oplus_j(\Omega^p_{{\mathbb P}^n}(\ell''_j)) \oplus \oplus_k(\Omega^q_{{\mathbb P}^n}(\ell'''_k))$
for some $\ell$, $\ell'_i$, $\ell''_j$ and $\ell'''_k$.
\end{lem}

\begin{proof}
By twisting if necessary, we may assume that ${\rm reg}\, {\mathcal E} = 1$.
In other words, $a_i({\mathcal E}) \le -i$ for any $i \ge 1$,
and $a_i({\mathcal E}) = -i$ for some $i = 1,n-1,n$.
\medskip
%The following remarks consider $a_3({\mathcal E}) = -3$ and
%$a_2({\mathcal E}) = -2$ cases, and we have only to study
%the case $a_1({\mathcal E}) = -1$ and $a_2({\mathcal E}) = -3$.

\noindent
{\bf Case I}.
If $a_n({\mathcal E}) = -n$, that is, ${\rm H}^n({\mathcal E}(-n)) \ne 0$,
then ${\mathcal O}_{{\mathbb P}^n}(-1)$ is 
a direct summand of ${\mathcal E}$ by Lemma~\ref{ACM}. Thus we have
a quasi-Buchsbaum bundle ${\mathcal E}'$ such that
${\mathcal E} \cong {\mathcal E}'\oplus {\mathcal O}_{{\mathbb P}^n}(-1)$.
Then we may reduce to a vector
bundle ${\mathcal E}'$ of lower rank.
\medskip

\noindent
{\bf Case II}.
If $a_q({\mathcal E}) = -q$, then a nonzero element $s \in
{\rm H}^q({\mathcal E}(-q))$ can be lifted up to
${\rm H}^0({\mathcal E} \otimes \Omega_{{\mathbb P}^n}^{q \vee}(-q))$,
and the corresponding element $t \in {\rm H}^{n-q}({\mathcal E}^{\vee}(-n-1+q))$
by Serre duality can be also lifted up to ${\rm H}^0({\mathcal E}^{\vee}
\otimes {\Omega}^q_{{\mathbb P}^n}(q))$.
Indeed, as in the fourth proof of (\ref{Chang-Goto}), 
surjective maps
${\rm H}^0({\mathcal E} \otimes \Omega_{{\mathbb P}^n}^{q \vee}(-q))
\to {\rm H}^q({\mathcal E}(-q))$
and
${\rm H}^0({\mathcal E}^{\vee}
\otimes {\Omega}^q_{{\mathbb P}^n}(q)) \to {\rm H}^{n-q}({\mathcal E}^{\vee}(-n-1+q))$
give $f : \Omega_{{\mathbb P}^n}^{n-1}(n-1) \to
{\mathcal E}$ and $g : {\mathcal E} \to \Omega_{{\mathbb P}^n}^{n-1}(n-1)$ corresponding to $s$ and $t$ respectively.
Hence we see $\Omega_{{\mathbb P}^n}^q(q)$ is a direct summand of
${\mathcal E}$ because $g \circ f $ is an isomorphism.
Since ${\mathcal E} \cong {\mathcal E}'\oplus \Omega_{{\mathbb P}^n}^q(q)$
for some vector bundle ${\mathcal E}'$ by (\ref{BuchsbaumLemma}),
we may reduce to a vector
bundle ${\mathcal E}'$ of lower rank again.
\medskip

\noindent
{\bf Case III}.
If $a_1({\mathcal E}) = -1$ and $a_q({\mathcal E}) \le -q-2$, then 
a nonzero element $s \in
{\rm H}^p({\mathcal E}(-p))$ can be lifted up to
${\rm H}^0({\mathcal E} \otimes \Omega_{{\mathbb P}^n}^{p \vee}(-p))$,
and the corresponding element $t \in {\rm H}^{n-p}({\mathcal E}^{\vee}(-n-1+p))$
by Serre duality can be also lifted up to ${\rm H}^0({\mathcal E}^{\vee}
\otimes {\Omega}^p_{{\mathbb P}^n}(p))$ by the same way as in Case II.
Thus
we have a direct summand
$\Omega^1_{{\mathbb P}^n}(1)$ of ${\mathcal E}$ as desired.
\medskip

Hence by taking out direct summands above we have a vector
bundle ${\mathcal F}$ with
with
$a_p({\mathcal F}) = -p$, $a_q({\mathcal F}) = -q-1$ and
$a_n({\mathcal F}) \le -n-1$.
\end{proof}

\begin{prop}
\label{newmainth}
Let ${\mathcal E}$ be an indecomposable quasi-Buchsbaum  bundle on 
${\mathbb P}^n$.
Assume that ${\rm H}^i_*({\mathcal E}) = 0$, $2 \le i \le n-2$.
Then there is an exact sequence for some $m$.
\[ 0 \to  {\mathcal O}_{{\mathbb P}^n}(-1)^{\oplus} \to
\Omega^1_{{\mathbb P}^n}(n)^{\oplus} \oplus 
{\mathcal O}_{{\mathbb P}^n}^{\oplus}
\to {\mathcal E}(m)^{\vee} \to 0. \]
\end{prop}

\begin{proof}
From (\ref{observation}), we may assume
$a_1({\mathcal E}) = -1$, $a_{n-1}({\mathcal E}) = -n$
and $a_n({\mathcal E}) \le -n-1$ in order to investigate quasi-Buchsbaum bundles.

By (\ref{lem1}),  we have an exact sequence
$0 \to {\mathcal E} \to {\mathcal F} \to {\mathcal L} \to 0$
with ${\rm H}^1_*({\mathcal F}) = 0$ with
${\rm H}^1_*({\mathcal F}) = 0$ by (\ref{lem1}), where 
${\mathcal L} \cong {\mathcal L}' \oplus {\mathcal L}''$,
${\mathcal L}' = \oplus {\mathcal O}_{{\mathbb P}^n}(1)$,
${\mathcal L}'' = \oplus {\mathcal O}_{{\mathbb P}^n}(e_i)$ for
$e_i > 1$.
Since ${\rm H}^i_*({\mathcal F}) = 0$, $1 \le i \le n-2$ and
${\rm H}^{n-1}_*({\mathcal F}) \cong {\rm H}^{n-1}_*({\mathcal E})$, 
from the structure theorem (\ref{Chang-Goto})
we have an isomorphism ${\mathcal F} \cong {\mathcal F}' \oplus
{\mathcal F}'' \oplus {\mathcal M}$ by the structure theorem (\ref{Chang-Goto}),
where ${\mathcal F}' =
\oplus \Omega^{n-1}_{{\mathbb P}^n}(n)$,
${\mathcal F}''= \oplus \Omega^{n-1}_{{\mathbb P}^n}(f_j)$ 
for some $f_j > n$ and ${\mathcal M}$ is a direct sum of line bundles.
Then we have a commutative diagram of exact sequences
\[
\begin{array}{ccccccccc}
0  & \to &  {\mathcal L}'^{\vee} & \stackrel{\varphi}{\to}
 & {\mathcal F}'^{\vee} \oplus {\mathcal M}^{\vee}  
& \to & {\rm Coker}\, \varphi & \to & 0 \\
 & & \downarrow  & & \downarrow & & \downarrow & & \\
0 & \to & {\mathcal L}^{\vee}  & \to &  {\mathcal F}^{\vee}
 & \to & {\mathcal E}^{\vee}  & \to & 0, \\
\end{array}
\]
Note that ${\mathcal L}'^{\vee} = \oplus {\mathcal O}_{{\mathbb P}^n}(-1)
\to {\mathcal F}''^{\vee} \cong \oplus \Omega^1_{{\mathbb P}^n}(n+1-f_j)$ is a zero map
because $\Gamma(\Omega^1_{{\mathbb P}^n}(\ell)) = 0$
for $\ell \le 1$.
Thus ${\rm Coker}\, \varphi$ is a direct summand of ${\mathcal E}^{\vee}$.
Since ${\mathcal E}$ is indecomposable, we see that ${\mathcal F}'' = 0$,
which yields nonzero elements in ${\rm H}^{n-1}_*({\mathcal E})$ are
concentrated in degree $-n$. Similarly, we have ${\rm H}^1({\mathcal E}(\ell))
 = 0$ for $\ell \ne -1$.

Now we have an exact sequence
\[ 0 \to {\mathcal L}^{\vee} = \oplus {\mathcal O}_{{\mathbb P}^n}(-1) \to
 {\mathcal F}^{\vee} = \oplus \Omega^1_{{\mathbb P}^n}(n) \oplus {\mathcal M}^{\vee}
\to {\mathcal E}^{\vee} \to 0. \]
Then we will show that ${\mathcal M}^{\vee}$ has only ${\mathcal O}_{{\mathbb P}^n}$
as a direct summand. Clearly it has no summand of the form 
${\mathcal O}_{{\mathbb P}^n}(\ell)$, $\ell \le -1$. 
If it has a summand of the form ${\mathcal O}_{{\mathbb P}^n}(\ell)$,
$\ell \ge 1$, then ${\rm H}^1_*({\mathcal E})$ has an element in the cokernel
of a nonzero map $\Gamma_*({\mathcal O}_{{\mathbb P}^n}(-\ell))
\to \Gamma_*({\mathcal O}_{{\mathbb P}^n}(1))$ not annihilated
by ${\mathfrak m}$, which contradicts with the quasi-Buchsbaum property.
\end{proof}

\begin{rem}
Our result (\ref{newmainth}) gives a detailed description of
Chang \cite[Theorem 4]{Chang2000} 
by using the Castelnuovo-Mumford regularity.

Moreover,  this recovers the fact (\ref{ES-KR}) that
an indecomposable quasi-Buchsbaum  bundle ${\mathcal E}$
of rank $2$ on  ${\mathbb P}^3$
is isomorphic to a null-correlation bundle with some twist. Indeed,
since ${\mathcal E}^{\vee} \cong {\mathcal E} \otimes \wedge^2 {\mathcal E}^{\vee}$,
we can put $t = \dim_k {\rm H}^1_*({\mathcal E}) = \dim_k {\rm H}^2_*({\mathcal E})$.
Then the short exact sequence in (\ref{newmainth}) gives 
${\rm rank}\, {\mathcal E} \ge 2t$, which implies $t=1$ as desired.
\end{rem}

\begin{cor}
\label{mainth1}
Let ${\mathcal E}$ be an indecomposable quasi-Buchsbaum  bundle on 
${\mathbb P}^n$.
Assume that ${\rm H}^i_*({\mathcal E}) = 0$, $2 \le i \le n-2$
and $\dim_k {\rm H}^1_*({\mathcal E}) = \dim_k {\rm H}^{n-1}_*({\mathcal E}) =1$.
Then ${\mathcal E}$ is isomorphic to one of the follows with some twist:
\begin{itemize} 
\item[{\rm (i)}]
Null-correlation bundle with $n$ odd
\item[{\rm (ii)}]
${\displaystyle 0 \to {\mathcal O}_{{\mathbb P}^n}(-1) \stackrel{\psi}{\to}
\Omega_{{\mathbb P}^n}(1)
\oplus {\mathcal O}_{{\mathbb P}^n}^{\oplus n-2m} \to {\mathcal E}^{\vee} \to 0}$, \par
where $\psi$ is given by $\varphi:  {\mathcal O}_{{\mathbb P}^n} \to
\Omega_{{\mathbb P}^n}(2)$ of rank $2m$ for some $1 \le m < n/2$
in {\rm (\ref{skew-symmetric})}.
\end{itemize}
\end{cor}

\begin{proof}
We may assume that ${\mathcal E}$ is not Buchsbaum and ${\rm H}^1({\mathcal E}(-1))
\cong {\rm H}^{n-1}({\mathcal E}(-n)) \cong k$. From (\ref{newmainth}), we have an
exact sequnce $0 \to  {\mathcal E} \to \Omega^{n-1}_{{\mathbb P}^n}(n) \oplus
{\mathcal L'} \to
{\mathcal O}_{{\mathbb P}^n}(1) \to 0$, where ${\mathcal L'}$ is a direct sum of
line bundles, because
${\rm H}^1_*({\mathcal E}) \cong k(1)$ and  ${\rm H}^{n-1}_*({\mathcal E}) \cong
k(n)$. Then we have a short exact sequence
$ 0 \to {\mathcal O}_{{\mathbb P}^n} \to \Omega^1_{{\mathbb P}^n}(2) \oplus 
{\mathcal O}_{{\mathbb P}^n}(1)^{\oplus r} \to {\mathcal E}^{\vee}(1) \to 0$.
A map  
$\varphi : {\mathcal O}_{{\mathbb P}^n} \to \Omega^1_{{\mathbb P}^n}(2)$
in the exact sequence is classified as in (\ref{skew-symmetric}) according to
the ${\rm rank}\, \varphi = 2m$, explicitly $\varphi$ is written as
$(x_1, -x_0, x_3, -x_2, \cdots, x_{2m-1}, -x_{2m-2}, 0, \cdots,0)$ by coordinate change.
A map ${\mathcal O}_{{\mathbb P}^n} \to 
{\mathcal O}_{{\mathbb P}^n}(1)^{\oplus r}$ is defined by linear polynomials 
$f_i$, and $I = (x_0, \cdots, x_{2m-1}, f_1, \cdots, f_r)$ must be minimally
generated and have no zero-points in ${\mathbb P}^n$.
Thus we have $r = n-2m$.
\end{proof}

\section{Syzygy theoretic approach to quasi-Buchsbaum bundles 
on ${\mathbb P}^3$}

Our topic of this section is a syzygy theoretic approach to study
the structure of a class of quasi-Buchsbaum bundles.
In particular, we focus on a study of quasi-Buchsbaum bundles
 on ${\mathbb P}^3$
with $\dim_k {\rm H}^1_*({\mathcal E})=
\dim_k{\rm H}^2_*({\mathcal E})=1$
by constructing a free resolution arising from their non-vanishing intermediate
cohomologies.
We will describes quasi-Buchsbaum bundles on ${\mathbb P}^3$
and the corresponding skew-symmetric matrices
in a view of standard system of parameters
developed from the theory of Buchsbaum rings.
%The behavior of system of parameters in the spectral
%sequence plays an important role to give the structure of
%quasi-Buchsbaum bundles.
%As a result of (\ref{pseudo-Buchsbaum thm}),
%we describe Buchsbaum, pseudo-Buchsbaum and non-standard Buchsbaum
%in terms of standard systems of parameters.
\medskip

Let $S = k[x_0,\cdots,x_n]$ be a polynomial ring over a field $k$
with $\deg x_i = 1$, $i=0, \cdots,n$.
Let $M$ be a finitely generated graded $S$-module of $\dim M = d + 1$.
Assume that $M$ has a finite local cohomology, that is,
$\ell ({\rm H}^i_{\mathfrak m}(M)) < \infty$, equivalently $\widetilde{M}$ is locally
free on ${\rm Proj}\, S$.

\begin{df}[\cite{T}]
\label{standard}
Let $f_1,\cdots,f_e$ be a part of homogeneous system of parameters
for an $S$-module $M$. We call $f_1,\cdots,f_e$ as a standard system if
${\mathfrak q}{\rm H}^i_{\mathfrak m}(M/{\mathfrak q}_jM) = 0$
for all nonnegative integers $i$, $j$ with $i + j \le d$, where
${\mathfrak q}_j = (f_1,\cdots,f_j)$, $j = 0,\cdots,e$ and ${\mathfrak q} = 
{\mathfrak q}_e$.
\end{df}

\begin{prop}[{\rm \cite{T}}]
Let $y_0,\cdots,y_d$ be a homogeneous system of parameters for an $S$-module
$M$. Then $y_0,\cdots,y_d$ is standard if and only if
the natural maps from the Koszul cohomologies
$ {\rm H^i}((y_0,\cdots,y_d); M) \to {\rm H}^i_{\mathfrak m}(M)$ for surjective
for $0\le i \le d$
\end{prop}

\begin{rem}[{\rm \cite{SV,T}}]
\label{standard=Buchsbaum}
In general, a graded $S$-module $M$ is Buchsbaum if and only if
any homogeneous system of parameters for the $S$-module $M$ is standard.
\end{rem}

Now let us consider a quasi-Buchsbaum bundle ${\mathcal E}$ on ${\mathbb P}^n$.
Let $S = k[x_0,\cdots,x_n]$ with ${\mathfrak m} = (x_0,\cdots,x_n)$
and $M = \Gamma_*({\mathcal E})$.
Note that
$M$ has the property ${\rm H}^i_{\mathfrak m}(M) = 0$
for $i=0,1$ and
${\mathfrak m}{\rm H}^i_{\mathfrak m}(M) = 0$ for $2 \le i \le n$.
Then $x_0,\cdots,x_n$ is a homogeneous system of parameters for
a graded $S$-module $M$. 
Now we will give slightly different definition of a standard ideal
from usual commutative algebra, not considering ${\rm H}^0_{\mathfrak m}$
and ${\rm H}^1_{\mathfrak m}$
local cohomologies through saturation and sheafification. 

\begin{df}
\label{standard2}
A homogeneous ideal $I \subset S$ is called standard for $M$ if any part of
homogeneous system of parameters $y_1,\cdots,y_j (\in I)$ is standard for any
$j \le n-2$,
that is, $(y_1,\cdots,y_j) {\rm H}^i_{\mathfrak m}(M/(y_1,\cdots,y_j)M) = 0$ for 
$2 \le i \le n-j$.  
\end{df}

\begin{rem}
\label{standardrem}
A vector bundle ${\mathcal E}$ on ${\mathbb P}^n$ is Buchsbaum if and only if 
${\mathfrak m}$ is standard for $M =\Gamma_*({\mathcal E})$ in the sense of (\ref{standard2}).
\end{rem}

Before proceeding on detailed study of quasi-Buchsbaum bundles, we
state a useful lemma, which follows from \cite{MTrieste, M2019} but
not necessarily clearly mentioned.

\begin{lem}
\label{commalg}
Let $S = k[x_0,\cdots,x_n]$ be the polynomial ring. Let $M$
be a quasi-Buchsbaum graded $S$-module of $\dim M = n+1$.
Let $y_1, y_2$ be a part of linear system of parameters
for $M$. Let $i$ be an integer with $i \le n-1$. Then
$y_2{\rm H}^i_{\mathfrak m}(M/y_1M) = 0$ if and only if
$y_1{\rm H}^i_{\mathfrak m}(M/y_2M) = 0$.
\end{lem}

\begin{proof}
From the quasi-Buchsbaum property of $M$, we have the following
exact sequences
in terms of a vector bundle ${\mathcal E}$:
\[
\begin{array}{ccccccccc}
0 & \to &  {\rm H}^i_{\mathfrak m}(M)(-1) & \to 
&  {\rm H}^i_{\mathfrak m}(M/y_1M)(-1) & \to &  
 {\rm H}^{i+1}_{\mathfrak m}(M)(-2) & \to & 0 \\
& & \downarrow y_2 & & \downarrow y_2 & & \downarrow y_2 & \\
0 & \to &  {\rm H}^i_{\mathfrak m}(M) & \to 
&  {\rm H}^i_{\mathfrak m}(M/y_1M) & \to &  
 {\rm H}^{i+1}_{\mathfrak m}(M)(-1) & \to & 0 \\
\end{array}
\]
Then the corresponding spectral sequence map
\[ d^{-n-1,i+1}_2 : E^{-n-1,i+1}_2 = {\rm H}^{i+1}_{\mathfrak m}
(M)(-2) \to E^{-n,i}_2 
= {\rm H}^i_{\mathfrak m}(M)^{(n+1)n/2} \]
consists of the map
\[ \varphi_{y_1 \wedge y_2} :  {\rm H}^{i+1}_{\mathfrak m}(M)(-2)
\to   {\rm H}^i_{\mathfrak m}(M)  \]
coming from the snake map of the commutative diagram.
Thus the following are equivalent:
(i) $y_2{\rm H}^i_{\mathfrak m}(M/y_1M) = 0$, 
(ii) $y_1{\rm H}^i_{\mathfrak m}(M/y_2M) = 0$ and
(iii)
$\varphi_{y_1 \wedge y_2} :  {\rm H}^{i+1}_{\mathfrak m}(M)(-2)
\to   {\rm H}^i_{\mathfrak m}(M)$ is zero,
as desired.
\end{proof}

Now let us consider an exact sequence:
\[ L^{\bullet} : 0 \to {\mathcal O}_{{\mathbb P}^n}(-n) \to 
{\mathcal O}_{{\mathbb P}^n}(-n+1)^{n+1} \to \cdots
\to {\mathcal O}_{{\mathbb P}^n}(-1)^{(n+1)n/2} \to \Omega^1_{{\mathbb P}^n}(1) \to 0, \]
where $L^{-n-1} =  {\mathcal O}_{{\mathbb P}^n}(-n)$, 
$L^{-2} =  {\mathcal O}_{{\mathbb P}^n}(-1)^{(n+1)n/2}$ and $L^{-1} = 
\Omega^1_{{\mathbb P}^n}(1)$,
arising from the Koszul complex $K_{\bullet}(x_0,\cdots,x_n;S)$,
Let $C^{\bullet}$ be the \u{C}ech resolution $C^i = C^i({\mathfrak U}; {\mathcal E})$,
where $\{ {\mathfrak U} \} = \{ {\rm D}_+(x_i) | 0 \le i \le 3 \} $
 is an affine open covering of ${\mathbb P}^3$.
Then we have a double complex $C^{\bullet \bullet} = 
L^{\bullet} \otimes C^{\bullet}$.
\medskip

Let us consider the case $n=3$ for a while, and assume
$\dim_k{\rm H}^1_*({\mathcal E}) = 1$.
From the consideration of the previous section, we assume
that ${\rm H}^1({\mathcal E}(-1)) \cong k$ and
${\rm H}^2({\mathcal E}(-3)) \ne 0$.
\medskip

The double complex $C^{\bullet \bullet}$ gives
a spectral sequence 
$\{ E^{p,q}_r \}$ with $E^{-4,q}_1 = {\rm H}^q_*({\mathcal E}(-3))$ and so forth.
Since ${\mathcal E}$ is quasi-Buchsbaum, we have a map
\[ d^{-4,2}_2 : E^{-4,2}_2 = {\rm H}^2_*({\mathcal E}(-3)) \to E^{-2,1}_2 
= {\rm H}^1_*({\mathcal E}(-1))^6. \]

Let us explain the above map explicitly. In fact,
$d^{-4,2}_2 : {\rm H}^2_*({\mathcal E}(-3)) \to {\rm H}^1_*({\mathcal E}(-1))^6$
is written in the following diagram:
\[
\begin{array}{ccccccccc}
0 & \to &  {\rm H}^1_*({\mathcal E}(-2)) & \to 
&  {\rm H}^1_*({\mathcal E}(-2)|_{H}) & \to &   {\rm H}^2_*({\mathcal E}(-3)) & \to & 0 \\
& & \downarrow x_j & & \downarrow x_j & & \downarrow x_j & \\
0 & \to &  {\rm H}^1_*({\mathcal E}(-1)) & \to 
&  {\rm H}^1_*({\mathcal E}(-1)|_{H}) & \to &   {\rm H}^2_*({\mathcal E}(-2)) & \to & 0 \\
\end{array}
\]
from a short exact sequence $0 \to {\mathcal E}(-1) \stackrel{x_i}{\to}
 {\mathcal E} \to {\mathcal E}|_H \to 0$, where $H = \{ x_i = 0 \}$.
Since mutliplication maps
${\rm H}^1_*({\mathcal E}(-2)) \stackrel{x_j}{\to} {\rm H}^1_*({\mathcal E}(-1))$
and  ${\rm H}^2_*({\mathcal E}(-3)) \stackrel{x_j}{\to}  {\rm H}^2_*({\mathcal E}(-2))$
are zero, by snake lemma we obtain a map 
$\varphi_{x_i \wedge x_j} : {\rm H}^2_*({\mathcal E}(-3))
\to {\rm H}^1_*({\mathcal E}(-1))$, $i \ne j$, see \cite{MTrieste, M2019}.
In this way we see a map $d^{-4,2}_2$ is written as
\[ \Phi = \bigoplus_{0 \le i < j \le 3} \varphi_{x_i \wedge x_j} :
 {\rm H}^2_*({\mathcal E}(-3))
\to \bigoplus_{0 \le i < j \le 3}{\rm H}^1_*({\mathcal E}(-1)). \]

Then we have a skew-symmetrix
matrix $A = A_s = (a_{ij})$, $a_{ij} = \varphi_{x_i \wedge x_j}(s)$ for
$s \in {\rm H}^2_*({\mathcal E})$.
In particular, in case 
$\dim_k{\rm H}^1_*({\mathcal E}) = 1$,
a skew-symmetric matrix $A = A_{\Phi} = (a_{ij})$, where
$a_{ij} = \varphi_{x_i \wedge x_j}(1) \in k$ for $0 \le i < j \le 3$,
depends only on an equivalence class of coordinate change.

\begin{rem}
\label{skew-symmetric_matrix}
If $s (\ne 0) \in {\rm H}^2({\mathcal E}(-3))$ satisfies $A = A_s = O$,
then ${\mathcal E}$ has  $\Omega^1_{{\mathbb P}^3}$ as a direct summand
from the fourth proof of (\ref{Chang-Goto}). In case
$\dim_k{\rm H}^1_*({\mathcal E}) = 1$,
${\mathcal E}$ is Buchsbaum if and only if $A = O$.
%More interestingly, the ${\rm rank}\, A = 2$ case has two vector
%bundles ${\rm (i)}$
%and ${\rm (ii)}$ of (\ref{mainex}) with
%the different properties, see (\ref{rank2}).
\end{rem}

Now assume that $\dim_k {\rm H}^1_*({\mathcal E})=1$,
and study the ring-theoretic property of $M = \Gamma_*({\mathcal E})$
in terms of the matrix $A = A_{\Phi}$ from
$\Phi = \oplus \varphi_{x_i \wedge x_j}$.
What we are interested in is the ring-theoretic property of 
vector bundles ${\mathcal E}$ of two cases,
${\rm rank}\, A =2$ and ${\rm rank}\, A = 4$.
Note that ${\rm det}\, A = ({\rm pf}\, A)^2$, where ${\rm pf}\, A =
a_{01}a_{23} - a_{03}a_{12} + a_{02}a_{13}$ is the Pfaffian of the skew-symmetric
matrix $A$. Then
the rank of $A$ can be 2 or 4 depending on ${\rm pf}\, A$.

\begin{df}
Let ${\mathcal E}$ be a quasi-Buchsbaum bundle on
${\mathbb P}^3$ with $\dim_k {\rm H}^1_*({\mathcal E})=1$.
Let $A$ be the correponding skew-symmetric matrix for
a nonzero element $s \in {\rm H}^2_*({\mathcal E})$.
Then $s$ is called a nonstandard element if ${\rm rank}\, A = 4$.
\end{df}

\begin{df}
\label{nonstandard-Buchsbaum}
Let ${\mathcal E}$ be a quasi-Buchsbaum bundle on ${\mathbb P}^4$
with $\dim_k {\rm H}^1_*({\mathcal E})= \dim_k{\rm H}^2_*({\mathcal E})=1$.
Let $A$ be the correponding skew-sysmetric matrix. Then
${\mathcal E}$ is called pseudo-Buchsbaum if ${\rm rank}\, A = 2$, and
${\mathcal E}$ is called nonstandard-Buchsbaum if ${\rm rank}\, A = 4$.
\end{df}

Let us consider an spectral sequence
$\{ E^{p,q}_r \}$ with $E^{-4,q}_1 = {\rm H}^q_*({\mathcal E}(-3))$
from the double complex $C^{\bullet \bullet} = L^{\bullet} \otimes C^{\bullet}$,
where $L^{\bullet} : 0 \to {\mathcal O}_{{\mathbb P}^n}(-3) \to 
{\mathcal O}_{{\mathbb P}^n}(-2)^4
\to {\mathcal O}_{{\mathbb P}^n}(-1)^6 \to \Omega^1_{{\mathbb P}^3}(1) \to 0$
and $C^{\bullet}$ is a \u{C}ech resolution of ${\mathcal E}$.
In order to describe a map
$d^{-4,2}_2 : E^{-4,2}_2 = {\rm H}^2_*({\mathcal E}(-3)) \to E^{-2,1}_2 
= {\rm H}^1_*({\mathcal E}(-1))^6$,
in the diagram below, we have
a map from $u \in C^2({\mathcal E}(-3))$
with $\alpha (u) = d^{-3,1} (v)$
to $\beta(v)$.
\[
\begin{array}{cccccccccccc}
0 & \to & C^2({\mathcal E}(-3)) & \stackrel{\alpha}{\to}
 & C^2({\mathcal E}(-2))^4 & \to & C^2({\mathcal E}(-1))^6
 & \to & C^2({\mathcal E}\otimes \Omega^1_{{\mathbb P}^3}(1)) & \to & 0 \\
 & &  \uparrow  & & \uparrow d^{-3,1} & & \uparrow & & \uparrow &  & \\
0 & \to & C^1({\mathcal E}(-3)) & \to & C^1({\mathcal E}(-2))^4 & 
\stackrel{\beta}{\to} & C^1({\mathcal E}(-1))^6
 & \to & C^1({\mathcal E}\otimes \Omega^1_{{\mathbb P}^3}(1)) & \to & 0 \\
 & & \uparrow & & \uparrow & & \uparrow d^{-2,0} & & \uparrow &  & \\
0 & \to & C^0({\mathcal E}(-3)) & \to & C^0({\mathcal E}(-2))^4 & \to & C^0({\mathcal E}(-1))^6
 & \stackrel{\gamma}{\to} & C^0({\mathcal E}\otimes \Omega^1_{{\mathbb P}^3}(1)) & \to & 0 \\
 & & \uparrow & & \uparrow & & \uparrow & & \uparrow &  & \\
 &  & 0 &  & 0 &  & 0 &  & 0 &  & \\
\end{array}
\]
If $\Phi \ne 0$, that is, not Buchsbaum,
then we see
$\beta(v) \not\in d^{-2,0}(C^0({\mathcal E}(-1))^6)$.
As in the fourth proof of (\ref{Chang-Goto}), the exact sequences
{\small
\[
\begin{array}{ccccccccccccc}
0 & \to & {\mathcal E}(-3) & \to &
 {\mathcal E}(-2)^4 & & \to & & {\mathcal E}(-1)^6  & \to  &
{\mathcal E}(-3) \otimes \Omega_{{\mathbb P}^3}^{2 \vee} \,  (\cong {\mathcal E} \otimes
\Omega^1_{{\mathbb P}^3}(1))
& \to & 0 \\
 & & & & & \searrow & & \nearrow & & & & & \\
 & & & & & &   {\mathcal E}(-3) \otimes  \Omega_{{\mathbb P}^3}^{1 \vee}  & & & & & & \\
 & & & & & \nearrow & & \searrow & & & & & \\
 & & & & 0 & & & & 0 & & & &
\end{array}
\]
}

\noindent
give a composition of maps ${\rm H}^0({\mathcal E}(-3) \otimes 
{\Omega}^{2 \vee}_{{\mathbb P}^3}) \to {\rm H}^1
( {\mathcal E}(-3) \otimes  \Omega_{{\mathbb P}^3}^{1 \vee})
\to {\rm H}^2({\mathcal E}(-3))$, but 
$s (\ne 0) \in {\rm H}^2({\mathcal E}(-3))$ cannot be lifted to
an element of ${\rm H}^0({\mathcal E}(-3))$, which prevents us
from getting a map
 ${\Omega}^2_{{\mathbb P}^3} \to {\mathcal E}(-3)$.
\medskip

Now we state our main result of this section.

\begin{theorem}
\label{maincor}
Let ${\mathcal E}$ be a
quasi-Buchsbaum bundle on ${\mathbb P}^3$ with
$\dim_k{\rm H}^1_*({\mathcal E}) = 1$
and $\dim_k{\rm H}^2_*({\mathcal E}) \ge 1$.
If there is a nonstandard element $s \in {\rm H}^2_*({\mathcal E})$,
then ${\mathcal E}$ has a null-correlation bundle with some twist
as a direct summand.
\end{theorem}

\begin{theorem}
\label{mainth}
Let ${\mathcal E}$ be an indecomposable
quasi-Buchsbaum bundle on ${\mathbb P}^3$
with ${\displaystyle \dim_k {\rm H}^i_* ({\mathcal E}) = 1}$ for $i=1,2$.
Then ${\mathcal E}$ is nonstandard-Buchsbaum if and only if ${\mathcal E}$
is isomorphic to a null-correlation bundle with some twist.
\end{theorem}

The proof of (\ref{maincor}) is similar to that of (\ref{mainth}). So,
we will prove (\ref{mainth}) for simplicity.
Let us prepare some observations for starting the proof of (\ref{mainth}).

\begin{lem}
\label{reg0}
Let ${\mathcal F}$ be a vector bundle on ${\mathbb P}^n$.
If ${\rm reg}\, {\mathcal F} = 1$ and
$\Gamma ({\mathcal F}^{\vee}(-1)) \ne 0$, then ${\mathcal F}$ has a
direct summand ${\mathcal O}_{{\mathbb P}^n}(-1)$.
\end{lem}

\begin{proof}
By (\ref{CMregularity}), ${\mathcal F}(1)$
is globally generated, which gives a surjective map 
${\mathcal O}_{{\mathbb P}^n}^{\oplus} \to {\mathcal F}(1)$,
and then a nonzero
composite map ${\mathcal O}_{{\mathbb P}^n}^{\oplus} \to {\mathcal F}(1)
\to {\mathcal O}_{{\mathbb P}^n}$ from $\Gamma ({\mathcal F}^{\vee}(-1)) \ne 0$.
Thus we have
direct summand ${\mathcal O}_{{\mathbb P}^n}$ of  ${\mathcal F}(1)$.
\end{proof}

\begin{rem}
\label{beforeproof}
From (\ref{observation}), 
we may assume that ${\rm H}^2_*({\mathcal E}) \cong k(3)$
and ${\rm H}^1_*({\mathcal E}) \cong k(1)$ by twisting
if necessary.
Then there exists a nonzero element 
 $s \in {\rm H}^2({\mathcal E}(-3)) \cong k$.
Further, $a_3({\mathcal E}) \le -4$ and ${\rm reg}\, {\mathcal E} = 1$.
By (\ref{reg0}), we see $\Gamma ({\mathcal F}^{\vee}(-1)) = 0$
because ${\mathcal F}$ is indecomposable.
Since ${\mathcal E}^{\vee}$ has the cohomological property similar
as above by Serre duality, we also see $\Gamma ({\mathcal F}(-1)) = 0$
by (\ref{reg0}).
\end{rem}

\begin{proof}[{\bf Proof of Theorem~\ref{mainth}}]
A null-correlation bundle is nonstandard-Buchsbaum. In fact, we have
a rank 4 skew-symmetric matrix. So we need to show the converse.
\medskip

\noindent
{\bf Step I} : To construct a null-correlation bundle ${\mathcal N}$
from a skew-symmetric
matrix $A$ of a vector bundle ${\mathcal E}$.
\medskip
 
For a nonstandard-Buchsbaum bundle ${\mathcal E}$, we have
$\Phi (s) \in  {\rm H}^1({\mathcal E}(-1))^6 \cong k^6$,
which defines a skew-symmetric matrix $A$ of rank $4$.
Then we have a null-correlation bundle ${\mathcal N}$
defined by $A$ as a short exact sequence:
\[ 0 \to {\mathcal O}_{{\mathbb P}^3}(-1) \stackrel{A}{\longrightarrow} 
\Omega^1_{{\mathbb P}^3}(1)  \to {\mathcal N}^{\vee} (\cong {\mathcal N})
\to  0. \]

Now we will construct a (not minimal) free resolution of ${\mathcal N}^{\vee}$.
Here we use the word `free resolution' if 
the sequence maintains the exactness after taking $\Gamma_*(\, \cdot \,)$.
By a commutative diagram with exact rows and columns
\[
\begin{array}{ccccccccc}
 & & 0 & & 0 & &  & \\
 & & \uparrow & & \uparrow & &  & \\
0 & \to &  {\mathcal O}_{{\mathbb P}^3}(-1) & \stackrel{A}{\longrightarrow} 
& \Omega^1_{{\mathbb P}^3}(1) & \to &  {\mathcal N}^{\vee} (\cong {\mathcal N}) & \to & 0. \\
& & \uparrow & & \uparrow & & & \\
 & & {\mathcal O}_{{\mathbb P}^3}(-1) & \stackrel{\Phi}{\longrightarrow}
 & {\mathcal O}_{{\mathbb P}^3}(-1)^6 &  &  &  & \\
 & & \uparrow & & \uparrow & & & & \\
 &  & 0 &  & {\mathcal O}_{{\mathbb P}^3}(-2)^4 & \ & &  &  \\
 & &  & & \uparrow & &  & \\
 & &  & & {\mathcal O}_{{\mathbb P}^3}(-3) & & & \\
 & &  & & \uparrow & &  & \\
 & &  & &  0 & &  &  \\
\end{array}
\]
\noindent
we have a free resolution of ${\mathcal N}^{\vee}$ by taking a mapping cone
\[ \bar{L} :
0 \to {\mathcal O}_{{\mathbb P}^3}(-3) \to {\mathcal O}_{{\mathbb P}^3}(-2)^4 
\oplus {\mathcal O}_{{\mathbb P}^3}(-1)
\to {\mathcal O}_{{\mathbb P}^3}(-1)^6 \to {\mathcal N}^{\vee} \to 0, \]
which yields the minimal free resolution
\[ 0 \to {\mathcal O}_{{\mathbb P}^3}(-3) \to {\mathcal O}_{{\mathbb P}^3}(-2)^4
\to {\mathcal O}_{{\mathbb P}^3}(-1)^5 \to {\mathcal N}^{\vee} \to 0. \]

\noindent
This resolution connects short exact sequences:
\[ 0 \to {\mathcal O}_{{\mathbb P}^3}(-3) \to {\mathcal O}_{{\mathbb P}^3}(-2)^4
 \to
 \Omega^2_{{\mathbb P}^3}(1) \to 0 \]
\[ 0 \to \Omega^2_{{\mathbb P}^3}(1) \to
 {\mathcal O}_{{\mathbb P}^3}(-1)^5 \to {\mathcal N}^{\vee} \to 0. \]

\noindent
{\bf Step II} : To give a nonzero map $f : {\mathcal N} \to {\mathcal E}$.
\medskip

What we have to do is take a nonzero element 
$f \in {\rm H}^0({\mathcal E} \otimes {\mathcal N}^{\vee})$
such that $\varphi (f) = s (\ne 0)$, where
\[ \varphi : {\rm H}^0({\mathcal E} \otimes {\mathcal N}^{\vee}) \to
{\rm H}^1({\mathcal E} \otimes \Omega^2_{{\mathbb P}^3} (1))  \to
 {\rm H}^2({\mathcal E}(-3))
 \]
arising from an exact sequence
\[ 0 \to {\mathcal E}(-3) \to {\mathcal E}(-2)^4 \oplus {\mathcal E}(-1)
\to {\mathcal E}(-1)^6
\to {\mathcal E}\otimes{\mathcal N}^{\vee}  \to 0. \]

In order to lift an element $s \in  {\rm H}^2({\mathcal E}(-3))$,
we will explain the behavior of the element explicitly by
comparing a double complex 
$C^{\bullet \bullet} = L^{\bullet} \otimes C^{\bullet}$,
where $L^{\bullet} : 0 \to {\mathcal O}_{{\mathbb P}^n}(-3) \to 
{\mathcal O}_{{\mathbb P}^n}(-2)^4
\to {\mathcal O}_{{\mathbb P}^n}(-1)^6 \to \Omega^1_{{\mathbb P}^3}(1) \to 0$
and a double complex below
 $D^{\bullet \bullet} = \bar{L}^{\bullet} \otimes C^{\bullet}$
from a free resolution of ${\mathcal N}^{\vee}$, where
$ \bar{L} :
0 \to {\mathcal O}_{{\mathbb P}^3}(-3)  \to {\mathcal O}_{{\mathbb P}^3}(-2)^4 
\oplus {\mathcal O}_{{\mathbb P}^3}(-1)
\to {\mathcal O}_{{\mathbb P}^3}(-1)^6 \to {\mathcal N}^{\vee} \to 0$
as numbering $\bar{L}^{-4} =  {\mathcal O}_{{\mathbb P}^3}(-3)$ and so on.
Note that $\bar{L}$ is not minimal free but has a superfluous component
${\mathcal O}_{{\mathbb P}^3}(-1)$, which plays an important role.

Let us take $u \in  C^2({\mathcal E}(-3))$  as an element 
$s (\ne 0) \in {\rm H}^2({\mathcal E}(-3))$ which gives
$v \in C^1({\mathcal E}(-2))^4$ with $d^{-3,1}(v) = \alpha (u)$
and $\beta(v) \not\in d^{-2,0}(C^0({\mathcal E}(-1))^6)$
in chasing an element
in the double complex $C^{\bullet \bullet}$ before.
However, we have $(\beta - \Phi)(v) = d^{-2,0}(w)$ for some
$w \in C^0({\mathcal E}(-1))^6$ from the construction of $\Phi$.
Thus we have a cycle $\delta (w)$
of $C^0({\mathcal E}\otimes {\mathcal N}^{\vee})$, which
gives an element $g \in {\rm H}^0({\mathcal E} \otimes
{\mathcal N}^{\vee})$ lifted from $s \in {\rm H}^2({\mathcal E}(-3))$
as desired.

\[
\begin{array}{cccccccccccc}
0 & \to & C^2({\mathcal E}(-3)) & \stackrel{\alpha}{\to}
 & C^2({\mathcal E}(-2))^4 
\oplus C^2({\mathcal E}(-1))
& \to &C^2({\mathcal E}(-1))^6
 & \to & C^2({\mathcal E}\otimes {\mathcal N}^{\vee}) & \to & 0 \\
 & & \uparrow & & \uparrow & & \uparrow & & \uparrow &  & \\
0 & \to & C^1({\mathcal E}(-3)) & \to & C^1({\mathcal E}(-2))^4 
\oplus C^1({\mathcal E}(-1)) & 
\stackrel{\beta - \Phi}{\to} &C^1({\mathcal E}(-1))^6
 & \to & C^1({\mathcal E}\otimes {\mathcal N}^{\vee}) & \to & 0 \\
 & & \uparrow & & \uparrow & & \uparrow d^{-2,0} & & \uparrow &  & \\
0 & \to & C^0({\mathcal E}(-3)) & \to & C^0({\mathcal E}(-2))^4 
\oplus C^0({\mathcal E}(-1)) & \to &C^0({\mathcal E}(-1))^6
 & \stackrel{\delta}{\to} & 
C^0({\mathcal E}\otimes {\mathcal N}^{\vee}) & \to & 0 \\
 & & \uparrow & & \uparrow & & \uparrow & & \uparrow &  & \\
 &  & 0 &  & 0 &  & 0 &  & 0 &  & 
\end{array}
\]

Once again, let us consider a spectral sequence $\{ E_r^{p,q} \}$
from the double complex 
$D^{\bullet \bullet} = \bar{L}^{\bullet} \otimes C^{\bullet}$.
From the arugument above,
we have $E_1^{-4,1} = E_2^{-4,1} = E_3^{-4,1} = 
{\rm H}^2(({\mathcal E}(-3))$. Also,
we have $E_1^{-2,0} = {\rm H}^0({\mathcal E}(-1)) = 0$ by (\ref{beforeproof}).
On the other hand, the spectral sequence converges to zero, that is,
$E_{\infty}^{p,q} = 0$, because
the row sequences have the exactness.
Thus we have $E_1^{-1,0} = E_2^{-1,0} = E_3^{-1,0}
= \Gamma^0({\mathcal E}\otimes {\mathcal N}^{\vee})$,
and an isomorphsim 
\[ d_3^{-4,2} : E_3^{-4,2} = {\rm H}^2(({\mathcal E}(-3)) \cong k \to 
E_3^{-1,0} = \Gamma({\mathcal E}\otimes {\mathcal N}^{\vee}) \cong k. \]
In particular, a nonzero map $f : {\mathcal N} \to {\mathcal E}$ is
unique up to the scalar.
\bigskip

\noindent
{\bf Step III}. To give a nonzero map $g : {\mathcal E} \to \Omega^2(3)$
and to show a composite map $g \circ f$ induces
an isomorphism ${\mathcal N} \cong {\mathcal E}$.
\medskip

By Serre duality, we have
an element $s'(\ne 0) \in {\rm H}^1({\mathcal E}^{\vee}(-1))$
corresponding to $s(\ne 0) \in {\rm H}^2({\mathcal E}(-3))$.
Let us take a free resolution
\[ \bar{L}' :
0 \to {\mathcal O}_{{\mathbb P}^3}(-1)  \to {\mathcal O}_{{\mathbb P}^3}^4 
\to \Omega^2_{{\mathbb P}^3}(3) (\cong
\Omega^{1 \vee}_{{\mathbb P}^3}(-1)) \to 0 \]
as numbering $\bar{L}'^{-4} =  {\mathcal O}_{{\mathbb P}^3}(-1)$ and so on.
Then let us take a double complex
 $D'^{\bullet \bullet} = \bar{L}'^{\bullet} \otimes C^{\bullet}$
and consider a spectral sequence $\{ E_r^{p,q} \}$ by 
abusing the same notation as a spectral sequence without confusing.

\[
\begin{array}{cccccccccccc}
0 & \to & C^1({\mathcal E}^{\vee}(-1)) & \to & C^1({\mathcal E}^{\vee})^4 
 & \to & C^1({\mathcal E}^{\vee} \otimes \Omega^2_{{\mathbb P}^3}(3)) 
& \to & 0 \\
 & & \uparrow & & \uparrow & & \uparrow &  & \\
0 & \to & C^0({\mathcal E}^{\vee}(-1)) & \to & C^0({\mathcal E}^{\vee})^4 
 & \stackrel{h}{\to} & 
C^0({\mathcal E}^{\vee} \otimes \Omega^2_{{\mathbb P}^3}(3)) & \to & 0 \\
 & & \uparrow & & \uparrow & & \uparrow &  & \\
 &  & 0 &  & 0 &  & 0 &  & 
\end{array}
\]

Similarly as in Step II,
we see that $E_1^{-4,1} = E_2^{-4,1} = {\rm H}^1({\mathcal E}^{\vee}(-1))$,
$E_1^{-2,0} = \Gamma({\mathcal E}^{\vee} \otimes
 \Omega^2_{{\mathbb P}^3}(3))$ and 
$E_2^{-2,0} = \Gamma({\mathcal E}^{\vee} \otimes
 \Omega^2_{{\mathbb P}^3}(3)) / h(\Gamma(({\mathcal E}^{\vee})^4)$.
 Then we have an isomorhism
\[ d_2^{-4,1} : E_2^{-4,1} = {\rm H}^1(({\mathcal E}^{\vee}(-1)) \cong k \to 
E_3^{-1,0} = \Gamma({\mathcal E}^{\vee} \otimes
 \Omega^2_{{\mathbb P}^3}(3)) / h(\Gamma(({\mathcal E}^{\vee})^4). \]
Thus a nonzero element $s'(\ne 0) \in {\rm H}^1({\mathcal E}^{\vee}(-1))$
gives a (not necessarily unique) nonzero map $g : {\mathcal E} \to 
\Omega^2_{{\mathbb P}^3}(3)$ as $g \in \Gamma({\mathcal E}^{\vee} \otimes
 \Omega^2_{{\mathbb P}^3}(3))$.
\medskip

A null-correlation bundle ${\mathcal N}$ is a subbundle of 
$\Omega^2_{{\mathbb P}^3}(3)$, and have
$\Gamma({\mathcal N}) = \Gamma({\mathcal N}^{\vee}) = 0$.
By taking ${\mathcal E} = {\mathcal N}$, we have an isomorphism
${\rm H}^1(({\mathcal N}^{\vee}(-1)) \cong 
\Gamma({\mathcal N}^{\vee} \otimes
 \Omega^2_{{\mathbb P}^3}(3)) \cong k$.
\medskip

From a commutative diagram
%{\small
\[
\begin{array}{ccccccc}
{\rm H}^0({\mathcal N}^{\vee} \otimes {\mathcal E})
\otimes {\rm H}^0({\mathcal E}^{\vee} \otimes
\Omega^2_{{\mathbb P}^3}(3))
& \to  &
{\rm H}^0({\mathcal N}^{\vee} \otimes
 \Omega^2_{{\mathbb P}^3}(3)) & \cong &
{\rm H}^0({\mathcal N}^{\vee} \otimes {\mathcal N}) & \cong &
 {\rm H}^{0}({\cal O}_{{\mathbb P}^3})
\\
\downarrow & & & & & & \downarrow   \\
{\rm H}^2({\mathcal E}(-3)) \otimes
{\rm H}^1({\mathcal E}^{\vee}(-1)) 
  & 
\to  & & & & &
 {\rm H}^3({\cal O}_{{\mathbb P}^3}(-4) ), 
\\
\end{array}
\]
%}

\noindent
we have an isomorphism
$g \circ f$. 
Hence
 ${\mathcal N}$ is a direct summand of ${\mathcal E}$ as desired.
\end{proof}

Next, let us turn to study a pseudo-Buchsbaum bundle. What we need is
the correspondence between 
a map $\psi : {\mathcal O}_{{\mathbb P}^3}(-1) \to
\Omega^1_{{\mathbb P}^3}(1) \oplus  {\mathcal O}_{{\mathbb P}^3}^2$
constructing a vector bundle ${\mathcal F}$
and a skew-symmetric matrix $A=A_{\mathcal F}$.

\begin{prop}
\label{correspondence}
Let $A = (a_{ij})$ be a skew-symmetric $4\times 4$-matrix
of rank 2 defining a pseudo-Buchsbaum bundle ${\mathcal F}$ on
${\mathbb P}^3$
defined by the short exact sequence:
\[ 0 \to {\mathcal O}_{{\mathbb P}^3}(-1) \stackrel{\psi}{\longrightarrow} 
\Omega^1_{{\mathbb P}^3}(1) \oplus  {\mathcal O}_{{\mathbb P}^3}^2
\to {\mathcal F}^{\vee}  \to 0, \]
where $\psi = (a_{00}x_0+\cdots+a_{03}x_{03}, \cdots, a_{30}x_0+
\cdots + a_{33}x_3, x_0, x_1)$.
Then the corresponding skew-symmetric matrix $A_{\mathcal F}$
defined by $\varphi_{x_i\wedge x_j}$, $0 \le i < j \le 3$ is equal to $A$.
\end{prop}

\begin{proof}
First, by changing coordinate we write 
\[ \psi = (\Psi, x_0,x_1) :  {\mathcal O}_{{\mathbb P}^3}(-1) \to
\Omega^1_{{\mathbb P}^3}(1) \oplus  {\mathcal O}_{{\mathbb P}^3}^2
\subset   {\mathcal O}_{{\mathbb P}^3}^4 \oplus  {\mathcal O}_{{\mathbb P}^3}^2, \]
where $\Psi(1) = (a_{00}x_0+\cdots+a_{03}x_{03}, \cdots, a_{30}x_0+
\cdots + a_{33}x_3)$ for a skew-symmetric matrix $A=(a_{ij})$ of rank 2. 
Let us construct a free resolution of
${\mathcal F}^{\vee}$.
For a given short exact sequence defining ${\mathcal F}$,
we will take a free resolution obtained from the
Koszul complex in the middle column:
\[
\begin{array}{ccccccccc}
 & & 0 & & 0 & &  & \\
 & & \uparrow & & \uparrow & &  & \\
0 & \to &  {\mathcal O}_{{\mathbb P}^3}(-1) & \stackrel{\psi}{\longrightarrow} 
& \Omega^1_{{\mathbb P}^3}(1) \oplus  {\mathcal O}_{{\mathbb P}^3}^2
& \to &  {\mathcal F}^{\vee} & \to & 0. \\
& & \uparrow & & \uparrow & & & \\
 & & {\mathcal O}_{{\mathbb P}^3}(-1)  & \stackrel{\psi}{\longrightarrow}
 & {\mathcal O}_{{\mathbb P}^3}(-1)^6 \oplus  {\mathcal O}_{{\mathbb P}^3}^2
 &  &  &  & \\
 & & \uparrow & & \uparrow & & & & \\
 &  & 0 &  & {\mathcal O}_{{\mathbb P}^3}(-2)^4 & \ & &  &  \\
 & &  & & \uparrow & &  & \\
 & &  & & {\mathcal O}_{{\mathbb P}^3}(-3) & & & \\
 & &  & & \uparrow & &  & \\
 & &  & &  0 & &  &  \\
\end{array}
\]
By taking a mapping cone, we have a (not minimal)
free resolution of
${\mathcal F}^{\vee}$ which preserves the exactness
after taking $\Gamma_*(\cdot)$:
\[\bar{L}^{\bullet} :
 0 \to {\mathcal O}_{{\mathbb P}^3}(-3) \to {\mathcal O}_{{\mathbb P}^3}(-2)^4
\oplus {\mathcal O}_{{\mathbb P}^3}(-1) \to {\mathcal O}_{{\mathbb P}^3}(-1)^6 \oplus
{\mathcal O}_{{\mathbb P}^3}^2 \to {\mathcal F}^{\vee} \to 0. \]
Then we consider a double complex $\bar{C}^{\bullet\bullet}
= \bar{L}^{\bullet} \otimes C^{\bullet}$, where $C^{\bullet}$
is the \u{C}ech resolution.
%comparing $C^{\bullet\bullet} = L^{\bullet} \otimes C^{\bullet}$.
A double complex $\bar{C}^{\bullet\bullet}$ gives
a spectral sequence $\{ \bar{E}^{p,q}_r \}$:
\[ \bar{d}^{-4,2}_2 : \bar{E}^{-4,2}_2 = 
{\rm H}^2_*({\mathcal F}(-3)) \cong k(3) \longrightarrow \bar{E}^{-2,1}_2 
= {\rm H}^1_*({\mathcal F}(-1))^5 
\oplus {\rm H}^1_*({\mathcal F})^2
\cong k(1)^5 \oplus k^2. \]
Here we remark that a map $\bar{d}^{-3,1}_1$ gives
an injection ${\rm H}^1_*({\mathcal F}(-1)) \to {\rm H}^1_*({\mathcal F}(-1))^6$
in the direct summand, so
 $\bar{E}^{-2,1}_2 = {\rm H}^1_*({\mathcal F}(-1))^6
/{\rm H}^1_*({\mathcal F}(-1)) \cong k(1)^5$.
%By thinking of degrees we need not concern the map to
%a component ${\rm H}^1_*({\mathcal F})^2$.
\medskip

On the other hand,
a spectral sequence  $\{ \bar{E}^{p,q}_r \}$ converges to $0$ because
of the exactness of $\bar{L}^{\bullet}$. This implies that
$\bar{E}^{-1,0}_4 = \cdots = \bar{E}^{-1,0}_{\infty} = 0$, while
$\bar{E}^{-1,0}_1 = \Gamma_*({\mathcal F}^{\vee} \otimes {\mathcal F})$.
\medskip

Let us study a map in degree zero
$\bar{d}^{-2,0}_1 : \bar{E}^{-2,0}_1(\cong \Gamma({\mathcal F})^2)
\to \bar{E}^{-1,0}_1 (\cong \Gamma({\mathcal F}^{\vee} \otimes {\mathcal F}))$
because $\Gamma({\mathcal F}(-1))=0$.
Then we will show $1 \in \Gamma({\mathcal F}^{\vee} \otimes {\mathcal F})
= {\rm Hom}({\mathcal F},{\mathcal F})$ is not in the image of
$\bar{d}^{-2,0}_1$. Indeed, if $1 \in {\rm Im}\, \bar{d}^{-2,0}_1$,
the element 1 would be lifted by
${\rm Hom}({\mathcal F}, {\mathcal O}_{{\mathbb P}^3}^2) \to
 {\rm Hom}({\mathcal F},{\mathcal F})$. This means that
${\mathcal F}$ is a direct summand of ${\mathcal O}_{{\mathbb P}^3}^2$,
which contradicts with the assumption.
Thus we see that $\bar{d}^{-2,0}_1$ is not surjective in degree 0.
Since  $\bar{E}^{-3,1}_4 = {\rm H}^1({\mathcal F}(-2)) = 0$ in degree 0 and
$\bar{E}^{-1,0}_4 = 0$, we see that
$\bar{d}^{-4,2}_3 : \bar{E}^{-4,2}_3(\cong k)  \to \bar{E}^{-1,0}_3$ must be
nonzero in degree 0. 
Hence $\bar{d}^{-4,2}_2 : \bar{E}^{-4,2}_2 \to \bar{E}^{-2,1}_2$
is a zero map.
\medskip

As described before, a spectral sequence $\{ E^{p,q}_r \}$ of a double complex
$C^{\bullet \bullet} = L^{\bullet} \otimes C^{\bullet}$,
where ${\displaystyle L^{\bullet} : 0 \to {\mathcal O}_{{\mathbb P}^3}(-3) \to 
{\mathcal O}_{{\mathbb P}^3}(-2)^4 \to
{\mathcal O}_{{\mathbb P}^3}(-1)^6 \to \Omega^1_{{\mathbb P}^3}(1) \to 0}$,
and $C^{\bullet}$ is \u{C}ech resolution of ${\mathcal F}$
gives a map
\[ d^{-4,2}_2 : E^{-4,2}_2 = {\rm H}^2_*({\mathcal F}(-3)) \cong k \to E^{-2,1}_2 
= {\rm H}^1_*({\mathcal F}(-1))^6 \cong k^6, \]
written as
$d^{-4,2}_2 = \oplus_{0 \le i < j \le 3} \varphi_{x_i \wedge x_j}$
defining a skew-symmetric matrix $A = A_{\mathcal F}$.
On the other hand, a map $\bar{d}^{-4,2}_2$ in the spectral sequence
$\{ \bar{E}^{p,q}_r \}$ gives a map 
${\rm H}^2({\mathcal F}(-3)) (\cong k) \to
{\rm H}^1_*({\mathcal F}(-1))^5 
\cong {\rm H}^1_*({\mathcal F}(-1))^6
/{\rm H}^1_*({\mathcal F}(-1)) (\cong k^5)$ which must be zero.
From the construction of mapping cone, we obtain
$\oplus_{0 \le i < j \le 3} \varphi_{x_i \wedge x_j}(1)
= (a_{ij})_{0 \le i < j \le 3}$, where $A = (a_{ij})$
defines ${\mathcal F}$. Thus we have the correspondence.
\end{proof}

Now we will state a classification theorem for psuedo-Buchsbaum
bundles.

\begin{theorem}
\label{pseudo-Buchsbaum thm}
Pseudo-Buchsbaum vector bundles ${\mathcal E}$ are isomorphic to
${\rm Coker}\, \psi$,
where $\psi : {\mathcal O}_{{\mathbb P}^3} \to \Omega^1_{{\mathbb P}^3}(2)
\oplus {\mathcal O}_{{\mathbb P}^3}(1)^2$,
classified $\psi$ as either of them:
\begin{enumerate}
\item[{\rm (i)}]
$\psi(1) = (x_3, x_3, x_3, -x_0-x_1+x_2) \times (x_0, x_1)$.
\item[{\rm (ii)}]
$\psi(1) = (0, x_3, x_3, -x_1+x_2) \times (x_0, x_1)$.
\item[{\rm (iii)}]
$\psi(1) = (0, 0, x_3,-x_2) \times (x_0, x_1)$.
\end{enumerate}
Here the corresponding skew-symmetric matrices are
\medskip

{\rm (i)}
${\displaystyle   \left(
\begin{array}{cccc}
0 & 0 & 0 & 1 \\
0 & 0 & 0 & 1 \\
0 & 0 & 0 & 1 \\
-1 & -1 & -1 & 0
\end{array}
\right)  \quad  }$
{\rm (ii)}
${\displaystyle   \left(
\begin{array}{cccc}
0 & 0 & 0 & 0 \\
0 & 0 & 0 & 1 \\
0 & 0 & 0 & 1 \\
0 & -1 & -1 & 0
\end{array}
\right)  \quad   }$
{\rm (iii)}
${\displaystyle   \left(
\begin{array}{cccc}
0 & 0 & 0 & 0 \\
0 & 0 & 0 & 0 \\
0 & 0 & 0 & 1 \\
0 & 0 & -1 & 0
\end{array}
\right)  \quad   }$
\end{theorem}

\begin{proof}
Let $A=(a_{ij})$ be the corresponding skew-symmetrix matrix of ${\mathcal F}$
We may assume $a_{01} = a_{02} = 0$.
by changing coordinates $x_1$ and $x_2$
as $x_1 - (a_{01}/a_{03}) x_3$ and $x_2 -(a_{02}/a_{03}) x_3$ respectively
if $a_{03} \ne 0$.
Since ${\rm pf}\, A = -a_{03}a_{12} = 0$, we see that
either $a_{12} = 0$ or $a_{03} = 0$. Further if $a_{03} = 0$ and $a_{12} \ne 0$,
we may assume $a_{12} = 0$ by changing coordinates.
In the end, the corresponding skew-symmetric matrix is isomorphic to
either of (i), (ii) and (iii).
\end{proof}

\begin{rem}
\label{pseudo-Buchsbaum rem}
The skew-symmetric matrix corresponding to a pseudo-Buchsbaum
bundle has 3 types which have different ring-theoretic properties
under hyperplane sections.
\begin{enumerate}
\item[{\rm (i)}]
There are no hyperplanes $H \subset {\mathbb P}^3$ such that
${\mathcal E}|_H$ is a Buchsbaum bundle.
\item[{\rm (ii)}]
There is only one hyperplane $H= \{x_0 = 0\} \subset {\mathbb P}^3$ 
such that
${\mathcal E}|_H$ is a Buchsbaum bundle, that is, isomorphic to
$\Omega^1_H(1) \oplus
\Omega^1_H(2)$. 
\item[{\rm (iii)}]
For any hyperplane $H$ containing $\{x_0 = x_1 = 0\} (\subset {\mathbb P}^3)$,
${\mathcal E}|_H$ is a Buchsbaum bundle. 
\end{enumerate}
\end{rem}

As a result of (\ref{pseudo-Buchsbaum thm}),
we describe Buchsbaum, pseudo-Buchsbaum and non-standard Buchsbaum
in terms of standard systems of parameters.

\begin{cor}
\label{standardsop}
Let $S = k[x_0.x_2,x_3,x_4]$ be the polynomial ring.
Let ${\mathcal E}$ be a quasi-Buchsbaum bundle on ${\mathbb P}^3
= {\rm Proj}\, S$ with $\dim_k {\rm H}^1_*({\mathcal E})=
\dim_k{\rm H}^2_*({\mathcal E})=1$.
Let $M = \Gamma_*({\mathcal E})$ be a graded $S$-module with
${\rm H}^i_{\mathfrak m}(M) = 0$, $i=0,1$.
Then there is a standard ideal generated by two linearly independent
elements. Furthermore, we have the following:
\begin{itemize}
\item[{\rm (i)}]
${\mathcal E}$ is Buchsbaum if and only if ${\mathfrak m}$
is a standard ideal, that is, 
${\mathfrak m}{\rm H}^2_*({\mathcal E}|_H) = 0$ for any hyperplane
$H \subset {\mathbb P}^3$
\item[{\rm (ii)}]
${\mathcal E}$ is pseudo-Buchsbaum if and only if
${\mathfrak m}$ is not standard and
there is a standard ideal generated by three linearly independent
elements $y_1,y_2,y_3 \in S_1$
such that
$(y_1,y_2,y_3){\rm H}^2_*({\mathcal E}|_{H_i}) = 0$ for any $i=1,2,3$.
\end{itemize}
\end{cor} 

Finally, we close this section 
by giving an example of a quasi-Buchsbaum bundle ${\mathcal E}$
 on ${\mathbb P}^3$ with $\dim_k {\rm H}^1_*({\mathcal E})=2$
and $\dim_k {\rm H}^2_*({\mathcal E})=1$

\begin{ex}
\label{quasi-Buchsbaum_ex}
Let $\varphi_1 : {\mathcal O}_{{\mathbb P}^3} \to \Omega^1_{{\mathbb P}^3}(2)
(\subset {\mathcal O}_{{\mathbb P}^3}(1)^{\oplus 4})$
by $\varphi_1(1) = (x_1,-x_0,0,0)$ and $\varphi_2 : 
{\mathcal O}_{{\mathbb P}^3} \to \Omega^1_{{\mathbb P}^3}(2)
(\subset {\mathcal O}_{{\mathbb P}^3}(1)^{\oplus 4})$
by $\varphi_2(1) = (0,0,,x_3,-x_2)$. Then $\varphi = \varphi_1 + \varphi_2 :
 {\mathcal O}_{{\mathbb P}^3} \to \Omega^1_{{\mathbb P}^3}(2)
\oplus  \Omega^1_{{\mathbb P}^3}(2)$ gives ${\mathcal E} = {\rm Coker}\, \varphi$
as a vector bundle of rank $5$. We see that ${\mathcal E}$
is an indecomposable quasi-Buchsbaum but not Buchsbaum bundle with
${\rm H}^1_*({\mathcal E}) \cong k(2) \oplus k(2)$ and
${\rm H}^2_*({\mathcal E}) \cong k$. 
\end{ex}

\section{Some Consequences on vector bundles on ${\mathbb P}^n$ with small intermediate cohomologies}

So far, we have studied quasi-Buchsbaum bundle on ${\mathbb P}^n$
with nonvanishing intermediate cohomologies only in ${\rm H}_*^1$ and
${\rm H}_*^{n-1}$. In particular, we have classified quasi-Buchsbaum
bundles on ${\mathbb P}^3$ with  $\dim_k{\rm H}_*^1 = \dim_k{\rm H}_*^2 = 1$.
This arises the following question. 

\begin{ques}
\label{question}
Let $1 \le j_1 < j_2 \le n-1$. Is there a
classification of quasi-Buchsbaum bundles ${\mathcal E}$
on ${\mathbb P}^n$
with 
${\rm H}^i_*({\mathcal E}) = 0$ for $1 \le i \le n-1$, $i \ne j_1,j_2$?
How about assuming
$\dim_k{\rm H}^1_*({\mathcal E}) = \dim_k{\rm H}^2_*({\mathcal E}) = 1$?
\end{ques}

Now we will study $j_1 = 1$, $j_2=1$ case.
Let ${\mathcal E}$ be a vector bundle on ${\mathbb P}^n = {\rm Proj}\, S$ with
${\rm H}^1_*({\mathcal E}) \cong k(1)$, ${\rm H}^2_*({\mathcal E})
\cong k(3)^{\oplus}$ and ${\rm H}^i_*({\mathcal E}) = 0$, $3 \le i \le n-1$,
where $S = k[x_0,\cdots,x_n]$ is a polynomial ring.
As in the previous section, we have a map $\Phi$
through the corresponding spectral sequence:
\[ \Phi = \bigoplus_{0 \le i < j \le n} \varphi_{x_i \wedge x_j} :
 {\rm H}^2({\mathcal E}(-3))
\to \bigoplus_{0 \le i < j \le 3}{\rm H}^1({\mathcal E}(-1)). \]
For $s (\ne 0) \in {\rm H}^2({\mathcal E}(-3))$, we have a skew-symmetic
matrix $A = A_s = (a_{ij})$, where $a_{ij} = \varphi_{x_i \wedge x_j}(s)$. In particular,
in case $\dim_k{\rm H}^2({\mathcal E}(-3)) = 1$, we put $A = A_{\mathcal E} =(a_{ij})$
with $a_{ij} = \varphi_{x_i \wedge x_j}(1)$.

Let us study a map ${\mathcal O}_{{\mathbb P}^n}(-1)
 (\cong \Omega^n_{{\mathbb P}^n}(n)) \to \Omega^{n-2}_{{\mathbb P}^n}(n-2)$.
Now we consider an exact sequence arising from Koszul complex:
\[ 0 \to S(-2) \to S(-1)^{\oplus n+1} \to S^{\oplus c_2} \stackrel{d_1}{\to}
S(1)^{\oplus c_3} \stackrel{d_2}{\to} S(2)^{\oplus c_4},  \]
where ${\displaystyle r_i = \left(
\begin{array}{c} n+1 \\ i \end{array}
\right), i=2,3,4}$.
Then we see  $\Gamma(\Omega^{n-2}_{{\mathbb P}^n}(n-1)) = {\rm Im}\, d_1
= {\rm Ker}\, d_2$.
We will describe an element of $\Gamma(\Omega^{n-2}_{{\mathbb P}^n}(n-1))$
in order to give a map of 
${\rm Hom}({\mathcal O}_{{\mathbb P}^n}(-1)), \Omega^{n-2}_{{\mathbb P}^n}(n-2))$.
Note that 
${\rm Hom}(\Omega^i_{{\mathbb P}^n}(i)), \Omega^j_{{\mathbb P}^n}(j))
\cong \wedge^{i-j}\Gamma({\mathcal O}_{{\mathbb P}^n}(1))$, see, e.g.~\cite[Lemma~1]
{DSch}.
Let us put $K_1 = \oplus_{0 \le i < j \le n} S e_i \wedge e_j = S^{\oplus c_2}$
and
$K_2 = \oplus_{0 \le i < j < k \le n} S e_i \wedge e_j \wedge e_k = S^{\oplus c_3}$.
Then
we write
$d_1(s) = \sum_{0 \le i < j < k \le n}
(a_{jk}x_i - a_{ik}x_j + a_{ij}x_k) e_i \wedge e_j \wedge e_k \in K_2$
for an element $s = \sum_{0 \le i < j \le n} a_{ij}e_i\wedge e_j \in K_1$. 
In other words,
a map ${\mathcal O}_{{\mathbb P}^n} \to \Omega^{n-2}_{{\mathbb P}^n}(n-1)
\subset {\mathcal O}_{{\mathbb P}^n}(1)^{\oplus c_3}$ is
given by linear forms
$(a_{jk}x_i - a_{ik}x_j + a_{ij}x_k)$, $0 \le i < j < k \le n$.
For a generic $(a_{ij})$, $0 \le i < j \le n$, the corresponding map
${\mathcal O}_{{\mathbb P}^n} \to \Omega^{n-2}_{{\mathbb P}^n}(n-1)$
is injective as vector bundles.
In fact, by taking $a_{ij} = 1$, the system of linear equations 
$x_i -x_j + x_k = 0$, $1 \le i < j < k \le n+1$
has only a trivial solution. 

For a generic $A = (a_{ij})$ we define a vector bundle
${\mathcal N}_{12}$ as an exact sequence
\[ 0 \to {\mathcal O}_{{\mathbb P}^n}(-1) \to \Omega^{n-2}_{{\mathbb P}^n}(n-2)
\to {\mathcal N}_{12}^{\vee} \to 0. \]
Then we easily have ${\rm H}^1_*({\mathcal N}_{12}) \cong k(1)$,
${\rm H}^2_*({\mathcal N}_{12}) \cong k(3)$ and ${\rm H}^i_*({\mathcal N}_{12}) = 0$,
$3 \le i \le n-1$. Also, we have ${\mathcal N}_{12}$ is a subbundle of
$\Omega^2_{{\mathbb P}^n}(-1)$ and ${\rm Hom}({\mathcal N}_{12}, {\mathcal N}_{12})
\cong {\rm Hom}({\mathcal N}_{12}, \Omega^2_{{\mathbb P}^n}(-1)) \cong k$,
similarly given as in \cite[(I.4.2)]{OSS}.
By a commutative diagram with exact rows and columns
\[
\begin{array}{ccccccccc}
 & & 0 & & 0 & &  & \\
 & & \uparrow & & \uparrow & &  & \\
0 & \to &  {\mathcal O}_{{\mathbb P}^3}(-1) & \to
& \Omega^{n-2}_{{\mathbb P}^n}(1) & \to &  {\mathcal N}_{12}^{\vee}  & \to & 0. \\
& & \uparrow & & \uparrow & & & \\
 & & {\mathcal O}_{{\mathbb P}^n}(-1) & \to
 & {\mathcal O}_{{\mathbb P}^n}(-1)^{\oplus c_2} &  &  &  & \\
 & & \uparrow & & \uparrow & & & & \\
 &  & 0 &  & {\mathcal O}_{{\mathbb P}^n}(-2)^{\oplus n+1} & \ & &  &  \\
 & &  & & \uparrow & &  & \\
 & &  & & {\mathcal O}_{{\mathbb P}^n}(-3) & & & \\
 & &  & & \uparrow & &  & \\
 & &  & &  0 & &  &  \\
\end{array}
\]
\noindent
we have a free resolution of ${\mathcal N}^{\vee}$ by taking a mapping cone
\[ 
0 \to {\mathcal O}_{{\mathbb P}^n}(-3) \to {\mathcal O}_{{\mathbb P}^n}(-2)^{\oplus n+1} 
\oplus {\mathcal O}_{{\mathbb P}^n}(-1)
\to {\mathcal O}_{{\mathbb P}^n}(-1)^{\oplus c_2} \to {\mathcal N}^{\vee} \to 0. \]
Then a skew-symmetric matrix $A$ corresponding to ${\mathcal N}$ gives
a map ${\mathcal O}_{{\mathbb P}^3}(-1)  \to \Omega^{n-2}_{{\mathbb P}^n}(1)$
by $\{ a_{jk}x_i - a_{ik}x_j + a_{ij}x_k \}_{0\le i < j < k \le n}$. 

\begin{dfprop}
\label{main-dfprop}
Let ${\mathcal E}$ be a vector bundle on ${\mathbb P}^n = {\rm Proj}\, S$ 
having only intermediate cohomologies in ${\rm H}^1({\mathcal E}(-1)) (\cong k)$
and ${\rm H}^2({\mathcal E}(-3))$.
Let $A$ be the corresponding skew-symmetric $(n+1)$-matrix for
$s (\ne 0) \in {\rm H}^2({\mathcal E}(-3))$. Then we call
 $s \in {\rm H}^2({\mathcal E}(-3))$ as a nonstandard
element if the corresponding map $\varphi : 
{\mathcal O}_{{\mathbb P}^n}(-1) \to \Omega^{n-2}_{{\mathbb P}^n}(n-2)$
is injective as a vector bundle map.

Under this assumption, that is, if there is a nonstandard element, then 
a vector bundle 
${\mathcal E}$ has ${\mathcal N}_{12}$ as a direct summand,
where ${\mathcal N}_{12}$ is the dual of the cokernel of $\varphi$,
which is similarly shown as (\ref{mainth}).
\end{dfprop}

\begin{ques}
What about a quasi-Buchsbaum bundle ${\mathcal E}$ with 
${\rm H}^i_*({\mathcal E}) = 0$, $i \ne 0, p, p+1, n$ for $p$ with $1 \le p \le n-2$?
If there is a map $\varphi :
\Omega^{n-p+1}_{{\mathbb P}^n}(n-p+1) \to \Omega^{n-p-1}_{{\mathbb P}^n}(n-p-1)$
as a vector bundle, then ${\mathcal N}_{p,p+1} = ({\rm Coker}\, \varphi)^{\vee}$
has only intermediate cohomologies ${\rm H}^p$ and ${\rm H}^{p+1}$.
The mapping cone free resolution would similarly play an important role to extend
our results.
\end{ques}

\section{Vector bundles on multiprojective spaces}

Is there a vector bundle ${\mathcal E}$ on $X = {\mathbb P}^m \times {\mathbb P}^n$
satisfying that
${\rm H}^i(X, {\mathcal E} \otimes {\mathcal O}_X(\ell_1,\ell_2)) = 0$, $1 \le i \le m+n-1$
for any $(\ell_1,\ell_2) \in {\mathbb Z} \times {\mathbb Z}$?
In fact there are no such vector bundles obtained from the basic property of the Castelnuovo-Mumford
regularity on multiprojective space.

\begin{dfprop}[\cite{BM}]
\label{multiregdf}
A coherent sheaf ${\mathcal F}$ on
$X = {\mathbb P}^m \times {\mathbb P}^n$ is $0$-regular if 
%the following conditions are satisfied.
%\smallskip
${\rm H}^i(X,{\mathcal F}(j_1,j_2)) = 0$ for $i \ge 1, \, j_1 + j_2 = -i, \,
-m \le j_1 \le 0, \, -n \le j_2 \le 0 $.

Then 
${\mathcal F}|_{H \times {\mathbb P}^n}$ is $0$-regular on 
$H \times {\mathbb P}^n
(\cong {\mathbb P}^{m-1} \times {\mathbb P}^n)$
for a generic hyperplane $H$ of ${\mathbb P}^m$, and
${\mathcal F}(m_1,m_2)$ is $0$-regular for any $m_1 \ge 0$, $m_2 \ge 0$,
and 
${\mathcal F}$ is globally generated.
\end{dfprop}

Let us consider a vector bundle ${\mathcal E}$ on 
$X = {\mathbb P}^m \times {\mathbb P}^n$ without
intermediate cohomology.
Let $t$ be the minimal
integer  such that ${\mathcal F}={\mathcal E}(t,t)$ is $0$-regular. 
Since the intermediate cohomologies
vanishes, we see that
${\rm H}^{m+n}(X, {\mathcal F}(-m-1,-n-1)) \ne 0$.
By Serre duality we have ${\rm H}^0({\mathcal F}^{\vee}) \ne 0$, which
gives a nonzero map
$\varphi : {\mathcal F} \to {\mathcal O}_X$.
On the other hand,  a globally generated bundle ${\mathcal F}$
gives a surjective map  $\psi: {\mathcal O}_X^{\oplus}
\to {\mathcal F}$. Since $\varphi \circ \psi$ is a nonzero map, 
${\mathcal O}_X$ is a direct summand of ${\mathcal F}$.
However, we see
${\rm H}^m(X,{\mathcal O}_X(-m-1,0)) \ne 0$, which contradicts with
the assumption of ${\mathcal E}$.

The following result is an immediate consequence of \cite[(2.5)]{M1},
given by a syzygy theoretic approach implies that
an ACM bundle on a smooth quadric surface $Q = {\mathbb P}^1
\times {\mathbb P}^1 \subset {\mathbb P}^3$ is isomorphic
to a direct summand of 
${\mathcal O}_X$, ${\mathcal O}_X(-1,0)$, ${\mathcal O}_X(0,-1)$
with some twist, a special case of \cite{Kn}.

\begin{prop}
\label{multiprop}
Let ${\mathcal E}$ be a vector bundle on
$X = {\mathbb P}^{n} \times {\mathbb P}^{n}$.
Then the following conditions are equivalent:
\begin{enumerate}
\item[{\rm (a)}]
\begin{enumerate}
\item[{\rm (i)}]
${\rm H}^i(X,{\mathcal E}(\ell_1, \ell_2)) = 0$ for any 
$\ell_1, \ell_2 \in {\mathbb Z}$ with $|\ell_1-\ell_2| \le n$, and
$i=1,\cdots,n-1,n+1,\cdots,2n-1$.

\item[{\rm (ii)}]
${\rm H}^n(X,{\mathcal E}(\ell, \ell)) = 0$ for any $\ell \in {\mathbb Z}$.
\end{enumerate}
\item[{\rm (b)}]
A vector bundle
${\mathcal E}$ is isomorphic to a direct sum of line bundles of the form
${\mathcal O}_X(u,v)$ for some $u, v$ with $|u-v| \le n$.
\end{enumerate}
\end{prop}

\begin{ques}
Find a generalization of cohomological criteria of (\ref{multiprop}) 
for vector bundles on
${\mathbb P}^n \times \cdots \times {\mathbb P}^n$.
\end{ques}
%\bigskip

%\noindent
%{\bf Declarations}:
%The authors have no conflicts of interest to declare that are relevant to the content of
%this article.

{\small
\noindent
Kumamoto University \\
Faculty of Advanced Science and Technology \\
Kurokami 2-40-1, Chuo-ku,
Kumamoto 860-8555,
Japan \\
e-mail: cmiyazak@kumamoto-u.ac.jp
}

\end{document}